%% file: Zec-Expansions-2nd-Order-6-arXiv.tex
\documentclass[12pt]{article}
\usepackage{amsmath,amssymb,amsthm}%,mathtools}
\usepackage{tikz-cd,multirow,ifthen}
\usepackage{stmaryrd}

\usepackage{fouriernc,graphicx}
 \usepackage[a4paper, textwidth=13.5cm, marginpar=0pt]{geometry}

\newtheorem{theorem}{Theorem}[section]
\newtheorem{lemma}[theorem]{Lemma}
\newtheorem{prop}[theorem]{Proposition}
\newtheorem{cor}[theorem]{Corollary}

\newtheorem{deF}[theorem]{Definition}

\newtheorem{alg}[theorem]{Algorithm}

\newtheorem{example}[theorem]{Example}

\newtheorem{notation}[theorem]{Notation}
 
\newtheorem{prty}[theorem]{Properties}

\input{basic_layout_local}

\input{MathMacro}

\input{algebra}

\input{zec_local-2nd}

\newcommand{\ProofInd}{OFF}

\begin{document}

 %\ProofOn

 \subsection*{Title of manuscript}

Generalized \zec\ expansions 
of the 2nd order
 
 \subsection*{Author }
  Sungkon Chang

 \subsection*{Abstract}

We generalize \zec's theorem to second-order linear recurrences of the form $G_{k+2} = gG_{k+1} + hG_{k}$ with arbitrary, coprime initial values $(G_1, G_2)$. 
We analyze the asymptotic behavior of the count $\#R_G(X)$ of positive integers less than or equal to $X$ that have an expansion under the standard rule of expansions $\mathcal{E}$ associated with the recurrence.  Additionally, we provide sufficient conditions under which $G$ has the unique expansion property under the \roe. Finally, we completely characterize the $g$-golden ratio recurrence family when $G_1 = 1$ or $G_2 = 1$. Our approach directly investigates the algebraic structure of the expansions, simplifying and completing previous partial results.

\newpage
 
 \section{Introduction}\label{sec:introduction}
 We denote a sequence $\seq G$ of numbers simply by $G$ and let $\nat$ denote the set of positive integers. Let $F$ be the shifted \fib\ sequence given by $(F_1,F_2)=(1,2)$ and $F_{k+2} = F_{k+1}+F_k$ for all $\kN$. \zec's Theorem \cite{zec} states that each positive integer can be uniquely written as a sum of distinct, non-adjacent terms of $F$. For example, $100=F_{10} +F_5+ F_3 = 89+ 8 + 3$, which is the unique expansion of $100$ using distinct, non-adjacent terms of $F$. In this paper, we consider a generalization of \zec's Theorem for second-order linear recurrences with arbitrary initial values. We present an  example in Section \ref{sec:ZT-gen} and state our main results (Theorems \ref{thm:main-1}, \ref{thm:main-2}, and \ref{thm:g-golden}) in Section \ref{sec:results}.
 
\subsection{Generalizations}
\label{sec:ZT-gen}

Let $\nat_0$ denote the set of non-negative integers.
In this paper, a positive integer $n$ is said to {\it be \expanD\ by a sequence  $ G$ of positive integers}
if 
\begin{equation} \label{eq:expansion}
n=\sum_{k=1}^\infty a_k G_k 
\end{equation} 
where $a_k\in\nat_0$ for $k\in\nat$, and 
the expression of the sum is called an {\it expansion of $n$ by $ G$}.
If certain conditions are imposed on the coefficients $a_k$, 
the set of conditions is called a {\it rule of expansions}.
For \zec's Theorem, the rule of expansions consists of the conditions that 
for all $k\in\nat$,
\begin{equation} 
a_k\in\set{0,1}\quad \&\quad a_k=1\implies a_{k+1}=0.
\label{eq:zec-rule}
\end{equation}
We may call it {\it the \zec\ rule of expansions}.

\zec's Theorem consists of three components: the sequence, the rule of expansions, and the set of numbers \expanD\ by the sequence under the \roe. 
By varying each component,
we encounter several interesting questions;
see \cite{chang-2021,chang-2023,chang-2024,daykin,hamlin,mw}. 
Below, we introduce an  example that represents the work presented in this paper.

Let $G$ be a sequence of positive integers such that 
$G_{k+2}= G_{k+1}+ G_{k}$ for all $k\in\nat$ and 
$(G_1,G_2)=(1,3)$.
Then, we may ask   how many positive integers less than or equal to $X$ are expanded by 
$G$ under the \zec\ rule?
Notice that $(G_1,G_2,\dots)=(1,3,4,7,\dots)$, and 
\GGG{
 1=G_1,\ 3=G_2,\ 4=G_3,\ 5=G_1+G_3,\ 7=G_4,\ 8=G_1+G_4,\ 10=G_2 + G_4.
 }
 Indeed,  only seven positive integers less than or equal to $10$  are expanded by $G$
 under the \zec\ rule.
The table below lists several examples of the counts of such expansions for larger values of $X$: 
\begin{equation*} 
\begin{array}{| c | c | c | c |}
\hline
\VS{1.2em}
X & 10^2 & 10^3 & 10^4\\
\hline
\# & 72 & 724 & 7236 \\
\hline
\end{array}\  .
\end{equation*}
   
  By Theorem \ref{thm:main-1} below, the proportion of positive integers less than or equal to $X$ that have expansions by $G$ under the \zec\ rule approaches 
\begin{equation} 
 \frac{2\phi+1}{3\phi+1}\approx 0.7236 \label{eq:ratio-example}
\end{equation}
as $X\to\infty$ where $\phi:=(1+\sqrt 5)/2$ is the golden ratio.
In other words, the probability of  a randomly chosen positive integer having an expansion under the \zec\ rule is approximately $72\%$.
Interestingly, this limiting ratio still involves the golden ratio, a consequence of the \zec\ rule of expansions  \cite{chang-2023}.

 In this paper, we generalize the problem from this example to second-order linear recurrences
 \begin{equation}
G_{k+2}= g G_{k+1} + h G_{k} \label{eq:recurr}
\end{equation}
  with 
arbitrary initial values where $g$ and $h$ are positive integers; see Theorems \ref{thm:main-1},   \ref{thm:main-2}, and Theorem \ref{thm:g-golden}.
 While this generalized version of the problem was only partially answered in \cite{chang-2018}, we demonstrate here how to obtain a complete answer for second-order linear recurrences.

A similar problem for the third-order \fib\ recurrence, $G_{n+3} = G_{n+2} + G_{n+1} + G_n$, is considered in \cite{chang-2024}; we adopt a similar approach in the present paper. However, significant differences arise from the \lq\lq asymmetrical\rq\rq\ structure of the rule of expansions associated with general second-order recurrences. In particular, in Section \ref{sec:R(X)}, we introduce an approach to these formulas that is simpler than the one in \cite{chang-2018}, and we present formulas for second-order recurrences that are more refined than those for the third-order \fib\ recurrence.

If $h=1$, the recurrence \eqref{eq:recurr} is associated with {\it the generalized golden ratio} (Example \ref{exm:g-golden}). In Theorem \ref{thm:g-golden}, we provide a complete answer for this general case when either $G_1$ or $G_2$ equals $1$. The proof for the case $G_2=1$ requires a novel analytical approach.

The remainder of the paper is organized as follows. In Section \ref{sec:generalized-expansions}, we introduce our main results; to establish the necessary context, we begin by presenting the standard \roe\ associated with the recurrence \eqref{eq:recurr}. In Section \ref{sec:lex-order}, we introduce the \lex\ order on the rule of expansions, which provides the underlying structure that clarifies the problem. In Section \ref{sec:classification}, we completely classify expansions that are distinct yet yield the same value in $\nat_0$. The algorithm for finding all such expansions is introduced at the end of Section \ref{sec:classification-sub}, and three   examples are presented in Sections \ref{sec:exm-I}, \ref{sec:exm-II}, and \ref{sec:exm-III}. In Section \ref{sec:R(X)}, we develop formulas to estimate the counts of the various expansions classified in Section \ref{sec:classification}. We prove Theorem \ref{thm:main-1} at the end of Section \ref{sec:second-term}, and demonstrate how to apply these formulas to our three running examples in Section \ref{sec:the-exm}. One of these examples concerns the recurrence associated with the generalized golden ratios; the pattern observed there is formalized as Theorem \ref{thm:g-golden} in Section \ref{sec:results}. We prove this theorem in Section \ref{sec:g-golden}. Given the linear recurrence \eqref{eq:recurr}, we can associate each pair of initial values $(G_1,G_2)$ with the limiting proportion of expandable integers, as in ratio \eqref{eq:ratio-example} (e.g., $(1,3)\mapsto 0.7236$). Finally, in Section \ref{sec:ratio}, we conclude the paper by analyzing this limiting ratio as a function of the initial values $(G_1,G_2)$. Throughout the paper, proofs that follow directly from standard definitions or straightforward arguments are omitted, with brief comments provided where appropriate.
  
\section{Generalized \zec\ expansions}\label{sec:generalized-expansions}

In this section, we introduce the rule of expansions associated with second-order recurrences.
Our main results are presented in Section \ref{sec:results}. 
\begin{deF}\label{def:cf} \rm Let $\mathfrak N^\infty$ be the subset of the countable product $\nat_0^\infty$ consisting of infinite tuples $\ep = (\ep_1, \ep_2, \dots)$ with only finitely many nonzero entries. We denote the $k$-th entry of $\ep$ by $\ep_k$. If $\ep$ is not the zero tuple $\mathbf 0$, we denote by $\ord(\ep)$ the largest index $\ell$ such that $\ep_\ell\ne 0$, and we define $\ord(\ep)=0$ if $\ep = \mathbf 0$. If $\ell:=\ord(\ep)>0$, we can represent $\ep$ as the finite tuple $(\ep_1,\dots, \ep_\ell)$ or $(\ep_1,\dots, \ep_\ell,0,\dots,0)$. Conversely, any finite tuple of non-negative integers can be identified with an infinite tuple in $\mathfrak N^\infty$ by appending zeros; for example, $(1,3,4) \mapsto (1,3,4,0,0,\dots)$. \end{deF} 

\begin{notation}\rm Given a sequence $G$ of positive integers and a coefficient tuple $\ep\in \mathfrak N^\infty$, we define $\ep \cdot G:=\sum_{k=1}^\infty \ep_k G_k$. In this context, we refer to $\ep$ as a {\it \cf} and to $G$ as a {\it base sequence}.
 \end{notation}

\subsection{The rule of expansions $\cE$}
Let $g$ and $h$ be positive integers. 
In this paper, we formally introduce only the standard rule of expansions associated with the recurrence \eqref{eq:recurr}. A more general principle of constructing a rule of expansions is introduced in \cite{chang-2021}. 

\begin{deF}\label{def:Z-collection} \rm The finite tuples of non-negative integers of the following forms are called the {\it semi-proper blocks} associated with the recurrence \eqref{eq:recurr}: \GGG{ (c)\text{ where } 0\le c\le g-1;\quad (c,g)\text{ where } 0\le c\le h-1. } For example, if $(g,h)=(3,5)$, the \spb s are \GGG{ (0),\  (1),\  (2),\  (0,3),\ (1,3),\ (2,3),\ (3,3),\  (4,3). } Although the definition of {\it proper blocks} is given in \cite{chang-2021}, we do not use this concept in the present work.

Let $\cEp$ be the subcollection of $\mathfrak N^\infty$ consisting of tuples obtained by concatenating \spb s. For example, if $(g,h)=(3,5)$, the following tuples are elements of $\cEp$: \GGG{ (0,3,2,1,1),\ (2,2,3,3,3),\ (0,3,0,2,0,1). } Let $\cEm$ be the subcollection of $\mathfrak N^\infty$ consisting of tuples of the form \GGG{ (g,\ep_2,\ep_3,\dots,\ep_\ell) } where $(\ep_2,\ep_3,\dots,\ep_\ell)\in\cEp$. We define $\cE:=\cEp \cup \cEm$, and we refer to this collection of \cf s as {\it the (standard) rule of expansions associated with the recurrence \eqref{eq:recurr}}.
 \end{deF} 
 Every finite tuple in $\cEp$ can be uniquely decomposed into a concatenation of \spb s (Corollary \ref{cor:unique-decomp}).
For the remainder of the paper, $\cE$ will denote 
the standard rule of expansions associated with the recurrence
 (\ref{eq:recurr}) unless otherwise specified.

\begin{example}\label{exm:cE-base-h} \rm Let $(g,h)=(h-1,h)$ where $h\ge 2$. In this case, the collection $\cE$ defined in Definition \ref{def:Z-collection} consists of the tuples $\ep=(\ep_1,\dots,\ep_\ell)$ satisfying $\ep_k\le h-1$ for $1\le k\le \ell$. Thus, the collection $\cE$ corresponds to the standard rule for base-$h$ expansions.
 \end{example} 
 
 \begin{example}\label{exm:cE-g-golden} \rm Let $(g,h)=(1,1)$. Here, the collection $\cE$ corresponds to the \zec\ \roe; that is, $\cE$ consists of the tuples $\ep\in\mathfrak N^\infty$ such that $\ep_k\le 1$ for all $k\in\nat$, and $\ep_k=1$ implies $\ep_{k+1}=0$. \end{example}

\subsection{Expansions under $\cE$}

 \begin{deF} \label{def:expansions-eval} \rm Let $G$ be a sequence of positive integers, and let $\cT$ be a subcollection of $\NN^\infty$. 
 If $n=\ep\cdot G$ for some $\ep\in\cT$, then $n$ is said to {\it be expanded by $G$ under $\cT$}, and the expression $\ep\cdot G$ is called a {\it $\cT$-expansion by $G$}. 
 
 Let $\eval_G : \cE \to \nat_0$ be the evaluation function defined by $\ep\mapsto \ep \cdot G$. When no ambiguity can arise, we write $\eval(\ep)$ instead of $\eval_G(\ep)$. 
  If $\eval_G$ is injective, then $G$ is said to {\it have the unique expansion property} under $\cE$. 
   If $\eval_G : \cE \to \nat_0$ is bijective, then $G$ is called a {\it \funds\ for $\cE$}. \end{deF} As an illustration, let $G$ be the sequence defined by $G_k=2k-1$, and let $\cE$ be the rule for binary expansions (Example \ref{exm:cE-base-h}). The sequence $G$ does not have  \uniexp\   because $G_1+G_2+G_3 = G_5$ where both the LHS and RHS are $\cE$-expansions. Furthermore, the relation $G_1 + G_k = 2k$ for $k > 1$ implies that every positive integer except $n=2$ can be represented as an $\cE$-expansion using $G$.
 
   \begin{deF}\rm \label{def:3-recurrence}
   Let  $\cH$ denote the collection of sequences $G $ of positive integers 
such that 
$G$ satisfies the recurrence (\ref{eq:recurr}) and the terms $G_1$ and $G_2$ are distinct and coprime.
\end{deF}

\begin{theorem}[\cite{chang-2021}]

Let $H$ be the sequence in $\cH$ such that $(H_1,H_2)=(1,g+1)$. 
If $G$ is an increasing \funds\ for $\cE$, then 
$G=H$.
In particular, $H$ is an increasing \funds\ for $\cE$.

\end{theorem}

This theorem implies that there is a unique increasing 
\funds\ $H$ for $\cE$.
The bijectivity of $\eval_H$ is also established  in \cite{hamlin,mw}.

\begin{deF}\label{def:char-poly}
\rm
Let $\al$ and $\beta$ be the (real) zeros of the characteristic polynomial
$x^2-gx - h$ such that $\al>1$, and 
let $a_0$ and $b_0$ be the real coefficients such that 
$H_k = a_0 \al^k + b_0 \beta^k$ for all $k\in\nat$ where 
$H$ is the increasing \funds\ for $\cE$.
\end{deF}

The existence of the constant multiples $a_0$ and $b_0$ in Definition \ref{def:char-poly} 
follows from the fact that $\al \ne \beta$.

\begin{example}\label{exm:base-h}
\rm
Let $(g,h)=(h-1,h)$ where $h\ge 2$, and recall from Example \ref{exm:cE-base-h} that $\cE$ is the rule of base-$h$ expansions.
Then,  $\al =h$ and 
$\beta =-1$.
The increasing fundamental sequence is given by 
$H_k=h^{k-1}$, and 
 each positive integer is uniquely written as an $\cE$-expansion by $H$.
 This is a well-known fact for the base-$h$ expansions.
\end{example}

\begin{example}\label{exm:g-golden}
\rm Let $(g,h)=(g,1)$.
Then,  
$$\al =\tfrac12\left( g + \sqrt{g^2+4}\right)
\text{ and }
\beta =\tfrac12\left( g - \sqrt{g^2+4}\right).$$
  The algebraic integer $\al$ is called
a {\it generalized golden ratio}, and we call it 
{\it the $g$-golden ratio}.
The increasing \funds\ is 
given by
$$H_k=\frac1{\beta-\al}\big( (\beta-1)\al^k -(\al-1)\beta^k\big)$$ for $k\in\nat$. 
The collection $\cE$ is called {\it the (standard) rule of expansions for the $g$-golden ratio recurrence}. 
 
If $g=2$, then $(H_1,H_2,\dots)=(1,3,7,17, 41,99,\cdots)$.
The \spb s are $(0)$, $(1)$, and $(0,2)$, and
listed below are
examples of $\cE$-expansions by $H$:
\GGG{
10=\ep\cdot H,\ 50 = \delta \cdot H,\ 100 = \tau \cdot H,\\
  \ep=(0,1,1)\in\cEp,\ \delta=(2,0,1,0,1)\in\cEm,\ \tau =(1,0,0,0,0,1)\in\cEp.
}
\end{example}

\subsection{Results} \label{sec:results}
 
We consider an arbitrary sequence $G$  in $\cH$, and determine how many positive integers less than or equal to $X$
have $\cE$-expansions by $G$.
We develop a general theory for second-order linear recurrences and demonstrate calculations for several special cases.
 The approach introduced in this paper is sufficiently general to yield formulas for all cases of $(g,h)\in\nat^2$.

\begin{deF}\rm \label{def:R}
Given a sequence  $G $ \opi, let $R_G$ be the set of $\nN$  that  
  have $\cE$-expansions by $G$, and 
let $R_G(X) := \set{ n \in R_G : n\le X}$ for $X\in\nat$.
When there is no ambiguity, we denote $R_G$ and $R_G(X)$ simply by $R$ and $R(X)$, respectively.

Recall the zeros $\al$ and $\beta$ from Definition \ref{def:char-poly}.
Let $\ome:=1/\al$. 
Given $G\in\cH$, define 
\GGG{
r := \frac{(g+1)\al + h }{G_{2}\al + hG_1 }=
 \frac{(g+1) + h\ome }{G_{2} + hG_1 \ome}.
}
Let $\beta_0:=\abs{\beta}$, and define the error term $E(X)$, which appears in Theorem \ref{thm:main-1}, by
\begin{equation}
E(X) = \begin{cases}
 1 & \text{if }\beta_0<1,\\
 \ln(X) & \text{if }\beta_0=1,\\
  X^\gamma & \text{if }\beta_0>1
  \end{cases}
  \label{eq:E(X)}
\end{equation}
where 
$\gamma=\log_\al \beta_0<1$.
\end{deF}

\begin{theorem}\label{thm:main-1}
Let $G\in\cH$.  Then, there are computable integers $c_k\in\nat_0$ for $1\le k\le 5$ such that  
$\#R_G(X) =pX
+ O(E(X))$ where
\GGG{
p=\frac{r  }{1+\ome}\left(1+\ome- \sum_{k=1}^{5} c_k \ome^{k}\right) .
}
If  $G\in\cH$  has  \uniexp, then $c_k=0$ for all $1\le k\le 5$, i.e., $\#R_G(X) =r X+ O(E(X))$.

 Let 
 $$\dis\ell_0:=\frac{ g(G_2 + G_1)}
  							{gG_2 + hG_1}.$$
The sequence
 $G$ has  \uniexp\ if one of the following is satisfied:
 \begin{enumerate}
 \item $G_2 > gG_1$;
\item $G_2>\max\set{g,h-1}$; 
\item $G_1>g+h$, provided that 
$g\le h-1$;
\item $G_1>g(\ell_0+1)+h$,
provided that  $g\ge h$, 
e.g., $G_1\ge (g^2+gh+h^2)/h$.
\end{enumerate}

\end{theorem}

While the estimate $\#R_G(X) = rX + O(E(X))$ for the case where $G$ has the \uniexp\  can be proved using the techniques introduced in \cite{chang-2018}, we present a simpler proof here by focusing on properties of $\cE$.
The primary contribution of this work is the analysis of cases where  \uniexp\   is not satisfied. 

The general theory presented in Section \ref{sec:classification} treats the cases $g\ge h$ and $g<h$ separately. The following theorem provides concrete examples of these results for several small values of $g$ and $h$.

\begin{theorem}\label{thm:main-2}   
 Let $G\in \cH$, and let $ p$ be the positive constant defined in Theorem \ref{thm:main-1}.
\begin{enumerate}
\item Let $(g,h)=(2,1)$. Then, $G$ has \uniexp\ if and only if 
$(G_1,G_2)$ does not belong to the set
\GGG{
I:=\set{(1,2),(2,1),(3,2),(3,1),(4,1),(5,1)}.
}
If $(G_1,G_2)\in I$, then $p=1$, i.e., $R(X) = X + O(1)$,
and if $(G_1,G_2)\not\in I$, then $p=r$.

\item Let $(g,h)=(1,2)$. Then, $G$ has \uniexp\ if and only if 
$(G_1,G_2)$ does not belong to the set $I:=\set{(2,1),(3,1)}$.
If $(G_1,G_2)=(2,1)$, then $p=5/6$, and 
if $(G_1,G_2)=(3,1)$, then $p=11/16$.
Finally, if $(G_1,G_2)\not\in I$, then $p=r$.

\item Let $(g,h)=(1,3)$. Then, $G$ has \uniexp\ if and only if 
$(G_1,G_2)$ does not belong to the set $I:=\set{(2,1),(3,1),(4,1)}$.
If $(G_1,G_2)\in I$, then
\AAA{
(G_1,G_2)=(2,1) &\implies p = \tfrac1{13}(9+2\ome)\approx 0.7591;\\
(G_1,G_2)=(3,1) &\implies p = \tfrac1{87}(41+25\ome)\approx 0.5961;\\
(G_1,G_2)=(4,1) &\implies p = \tfrac1{153}(55+50\ome)\approx 0.5014.
} 
If $(G_1,G_2)\not\in I$, then $p=r$.

\end{enumerate}

\end{theorem}

The limiting ratio  $p$ for the case $(G_1,G_2)\in I$ in Part (1) of Theorem \ref{thm:main-2} is equal to $1$,   which distinguishes it from the other two cases.
   We prove the following  result  for this case:
\begin{theorem} \label{thm:g-golden}
Let    $G\in \cH$ for the $g$-golden ratio recurrence.
If $G_1=1$ and  $G_2\le g$, then $\#R_G(X)=X$, and
if  $G_1=1$ and  $G_2> g$, then $G$ has \uniexp, and $p=r<1$ as described in Theorem \ref{thm:main-1}.
If $G_2=1$ and $G_1\le g^2+1$, then
$\#R_G(X) = X + O(1)$, and 
if $G_2=1$ and $G_1> g^2+1$, then $G$ has \uniexp, and $p=r<1$ as described in Theorem \ref{thm:main-1}.

\end{theorem}  

If $(G_1,G_2)=(1,g+1)$, then $G=H$ is the \funds, so 
$\#R(X)=X$.
If $G_1=1$ and $G_2\ge g+2$, then 
Theorem \ref{thm:g-golden} implies that 
$$p=\frac{ (g+1)\al + 1 }{ G_2 \al + 1}
 <\frac{ (g+1)\al + 1 }{ (g+1)\al + 1} =1.$$
If $G_2=1$ and $G_1\ge g^2+2$, then 
the theorem implies that 
\GGG{
p=\frac{ (g+1)\al + 1 }{   \al + G_1}
 \le \frac{ (g+1)\al + 1 }{  \al +g^2+2} 
 =\frac{  \al +g\al+ 1 }{  \al +g^2+2} 
 =\frac{  \al +\al^2}{  \al +g^2+2}< 1.
 }
     
      \section{The \lex\ order}\label{sec:lex-order}

We introduce a \lex\ order on $\cE$, which is the domain  of $\eval_G$.
\begin{deF}
\rm \label{def:lex-order}
Let $\set{\ep,\delta}\subset \NN^\infty$. 
Define $\ep <\delta$ if there is an index $m$ such that 
$\ep_m<\delta_m$ and $\ep_k=\delta_k$ for all $k>m$, e.g.,
if $\ord(\ep)<\ord(\delta)$, then $\ep<\delta$.
We refer to this relation as 
{\it the \lex\ order on $\NN^\infty$.}
\end{deF}
Thus,   the \spb s are ordered as follows:
\begin{equation}
(0) < (1) <\cdots < (g-1) < (0,g) < (1, g) <
\cdots < (h-1,g).  \label{eq:spb-order}
\end{equation}

\begin{notation}
\rm
Let $\set{\vEc v,\vEc w}\subset \NN^\infty$.
 We denote by $ (\vEc v , \vEc w)\in\NN^\infty$ the finite tuple obtained by
 concatenating the two tuples, e.g., 
if $\vEc v=(1,2)$ and $\vEc w=(3,4,5)$, then 
$(\vEc v, \vEc w)= (1,2,3,4,5)$.
In general, $(\vEc v, \vEc w)$ does not denote the pair in $\NN^\infty \times\NN^\infty $.
Reading from right to left, 
we  use the vertical bar symbol $\mid$  instead of a comma to indicate the boundaries of the \spb s.
For example,
if $(g,h)=(3,5)$ and $
\ep= (3,2,2,3,3,3)$, 
  we write its \spb\ decomposition as 
 $$ \ep =  (3\mid 2\mid 2,3\mid 3,3).$$
   In general, we represent a \spb\ decomposition as 
\begin{equation}
\label{eq:block-decomp}
\ep =( \mathbf b_\ell\mid\mathbf b_{\ell-1}\mid \cdots \mid\mathbf b_1)
\end{equation}
where $\mathbf b_k$ is a \spb\ for each $1\le k\le \ell-1$, and $\mathbf b_\ell$ is either a \spb\ or the single-entry block $(g)$.
We use a colon  instead of a comma  to indicate that the suffix is a member of
$\cEp$, e.g.,
$\ep=(3,2: 2,3,3,3)$ indicating that  $(3,2)\in\cE $ and $(2,3,3,3)\in\cEp$.
Because concatenation is associative, the placement of the colons is not uniquely determined, e.g.,
$$(3,2: 2,3,3,3)=(3,2: 2,3:3,3) =(3:2, 2,3:3,3) =(3:2, 2,3,3,3) .$$
For convenience, 
 we define the empty tuple $\iota:=(\ )$,  
  $(\iota,\vEc w ) := \vEc w$, and $(\iota:\vEc w ) := \vEc w$.

\end{notation}

Lemma \ref{lem:block-decomp} can be proved by induction on the length $\ell$; we leave the details to the reader. 
\begin{lemma}\label{lem:block-decomp}
Let $\set{\ep,\delta}\subset\cE$ such that $u:=\ord(\ep)\le\ell:=\ord(\delta)$.
Let the decompositions of $\ep$ and $\delta$ into \spb s be given by
\AAA{
\ep &= (\mathbf b_n\mid\mathbf b_{n-1}\mid \cdots \mid\mathbf b_1),\\
\delta &= (\mathbf c_m\mid\mathbf c_{m-1}\mid \cdots \mid\mathbf c_1 )
}
where, if $u<\ell$, the tuple $\ep$ is padded with zeros to obtain a tuple of length $\ell$ before decomposition. 
Then,  $\ep<\delta$ if and only if   there is an index $p$ such that $\mathbf b_p<\mathbf c_p$ and 
$\mathbf b_k = \mathbf c_k$ for all $1\le k<p$.
 
\end{lemma}

\begin{cor}\label{cor:unique-decomp}
Every coefficient tuple $\ep\in\cE$ has a unique decomposition into \spb s.
\end{cor}

Under the \lex\ order, both $\NN^\infty$ and its subcollection $\cE$ are totally ordered.
As shown in Theorem \ref{thm:well-ordered} below, the \lex\ order on $\cE$ can be characterized via base-$B$ expansions  \cite[Lemma 37]{chang-2021}. For the remainder of this section, let $B:=1+\max\set{g,h-1}$ and define
 $\NNB:=\set{\ep\in\NN^\infty : \ep_k\le B-1 \text{ for all } k\in\nat}$. This collection corresponds to the rule of base-$B$ expansions, and the following results are well known: 

\begin{theorem}\label{thm:well-ordered}
Let  $Q$ be the sequence given by $Q_k= B^{k-1}$ for $\kN$.
Then, the evaluation function $\eval_B$ from $\NNB$ to $\nat_0$ given by $\ep \to \ep\cdot Q$ is a strictly increasing bijection.
\end{theorem}

\begin{cor}
\label{thm:finiteness}
Given $\ep\in \NNB$, the lower set $\set{\tau \in \NNB : \tau \le \ep}$ is finite.
\end{cor}

The next corollary is a direct consequence of the well-orderedness of $\nat_0$.

\begin{cor} \label{cor:well-ordered}
The collection $\NNB$ is well-ordered under the \lex\ order, i.e., 
 all non-empty subcollections $\cT$ of $\NNB$ have a smallest tuple $\ep\in \cT$.
\end{cor}

We next define the least upper bound in a subcollection $\cT$.
Its existence is guaranteed by Corollary \ref{cor:well-ordered}, and its uniqueness follows because the set is totally ordered.

\begin{deF}
\rm \label{def:imm-succ}
Let $\cT$ be an infinite subcollection of $\NNB$, and let $\ep\in\cT$.
By Corollary \ref{thm:finiteness}, the subcollection
$U:=\set{\delta\in \cT : \ep <\delta}$ is non-empty.
The smallest tuple in $U$, denoted by $\lub_{\cT}(\ep)$, is called {\it the   least upper bound of $\ep$ in $\cT$}.
Given $n\in\nat_0$, let $\lub_{\cT}^n(\ep)$ denote the $n$th iteration of $\lub_{\cT}$ on $\ep$.
 
\end{deF}

The following lemma shows how to find $\lub_\cT(\ep)$.
\begin{lemma}\label{lem:find-lub}
Let $\cT$ be an infinite subcollection of $\NNB$, and let 
$\ep\in \cT$.
Then, there is a smallest integer $n\in \nat$ such that $ \lub_{\NNB}^n(\ep)\in\cT$. Moreover, 
 $\lub_{\cT}(\ep) = \lub_{\NNB}^n(\ep)$. 
\end{lemma}

\begin{proof}
The infinitude of $\cT$ and Corollary \ref{thm:finiteness} imply that
there is $\tau\in \cT$  such that $\ep<\tau$.
Since $\NNB$ is totally ordered, there is $n\in\nat$ such that  $ \lub_{\NNB}^n(\ep)=\tau$.
Let $n$ be the smallest such integer, and let $\delta \in\cT$  
be any element satisfying
 $\ep<\delta$.
Then, there is $m\in\nat$ such that  $ \lub_{\NNB}^m(\ep)=\delta$.
By the minimality of $n$, we have $n\le m$, which implies $\tau \le \delta$ because $\lub_{\NNB}$ is strictly increasing (Theorem \ref{thm:well-ordered}).
\end{proof}

\begin{example}\rm \label{exm:lub}
Let $(g,h)=(2,1)$, so $B=3$.
Let $\ep=(2,0,2,0,2,0,1,1)\in\cE$.
By Theorem \ref{thm:well-ordered}, we have
\GGG{
 \lub_{\NNB}(\ep) =(0,1,2,0,2,0,1,1)
 \implies \lub^3_{\NNB}(\ep) = (2,1,2,0,2,0,1,1),\\
 \lub_{\NNB}^4(\ep) =(0,2,2,0,2,0,1,1)
 \implies \lub_{\NNB}^{60}(\ep) = (2,2,2,2,2,0,1,1)\\
 \implies
 \lub_{\NNB}^{61}(\ep) = (0,0,0,0,0,1,1,1)\in\cE .
}
 Then, Lemma \ref{lem:find-lub} yields $\lub_\cE(\ep) = (0,0,0,0,0,1,1,1)$ since $\lub_{\NNB}^k(\ep)\not\in \cE$ for
 $1\le k\le 60$.
 
More generally, we observe that the longest prefix of the form 
$$(h-1,g:h-1,g:\dots:h-1,g)\quad\text{or}\quad
 (g:h-1,g: \dots:h-1,g),$$
 if exists, is replaced with $(0,\dots,0)$, with
 a carry of $1$ added to the next higher entry when computing $\lub_{\cE}(\ep)$.
 That is, if $\delta=(0,0,2,0,2,0,1,1)$, then
 \GGG{
 \lub_\cE(\delta)=(1,0,2,0,2,0,1,1)
 \to
 \lub_\cE^2(\delta)=({\bf 2,0,2,0,2},0,1,1)\\
 \to
 \lub_\cE^3(\delta)=(0,0,0,0,0,1,1,1).
 }
 
 \end{example}

This example demonstrates that
$$
\vEc w=(h-1,g,h-1,g,\dots,h-1,g:\vEc w') \in\cEp
\implies
G_1 + \vEc w \dG = \lub_\cE( \vEc w )\dG$$
since $G$ satisfies the recurrence (\ref{eq:recurr}).
In fact, this relation holds for all  $\vEc w\in\cEp$.
Recall all the \spb s from (\ref{eq:spb-order}), and let $\vEc w =(\vEc b_\ell \mid \cdots \mid \vEc b_1 )\in\cEp$ be the 
\spb\ decomposition.
If $\vEc b_\ell<(h-1,g)$, then it is straightforward to verify that
$$G_1 + \vEc w \dG = (1+\vEc w_1,\vEc w_2,\dots)\dG  
= \lub_\cE( \vEc w )\dG.$$ 
Using these observations, Lemma \ref{lem:add-Gn} can be established directly; we omit its proof.
\begin{lemma} \label{lem:add-Gn}
Let $\ep=(0,\dots,0)$ be the finite zero tuple with $n-1$ entries where $n\ge 1$  and 
$\vEc w\in \cEp$, e.g., if $n=1$, then $\ep$ is the empty tuple.
Then, 
\GGG{
	G_n +(\ep : \vEc w ) \dG =  ( \ep , \lub_\cE( \vEc w ) ) \dG .
}
Suppose that $\vEc w = (\vEc w_1 : \vEc v)$ for $\vEc v \in \cEp$.
If $\vEc w_1<g-1$ and $n\ge 2$, then the RHS is $(\ep:\lub_\cE( \vEc w ) ) \dG$,
and if $\vEc w_1=g-1$ and $n\ge 2$, then the RHS is $(\ep_1,\dots,\ep_{n-2}: 0,g: \vEc v ) \dG$
where  $\ep_k=0$ for all $1\le k\le n-2$.

\end{lemma}

\section{Classification of non-unique expansions}\label{sec:classification}

 Throughout this section, let $G$ be a sequence in $ \cH$.
 We completely classify the tuples $\ep\in\cE$, for which there is a tuple $\delta>\ep$ in $\cE$
 with $\eval_G(\ep) = \eval_G(\delta)$, i.e., $\ep \dG = \delta \dG$.

 \subsection{Unique expansions over a subcollection of $\cE$}
 
In Theorem \ref{thm:partial-unique} below, we prove that the restriction  of $\eval_G$ to the
 subcollection of $\cE$, defined in Definition \ref{def:cF0} below, is 
injective.
This observation is instrumental in classifying the pairs $\set{\ep,\delta}\subset\cE$ such that 
$\ep \dG = \delta \dG$.

\begin{theorem}
\label{thm:initial-unique}
Suppose that $gG_1<G_2$.
Let $\set{\ep,\delta}\subset \cE$. Then, $\ep<\delta$ if and only if $\ep\dG < \delta\dG $.

\end{theorem} 

\begin{proof}
We first show that  $\ep\dG < \lub_\cE(\ep) \dG$.
Suppose that $\ep\in\cEp$. 
Then, by Lemma \ref{lem:add-Gn}, we have
$ \ep\dG<G_1 + \ep\dG = \lub_\cE(\ep)\dG$.
Next, suppose that 
$$\ep=(g:\vEc w )\in \cEm$$ where $\vEc w \in \cEp$, 
and let $G_0:=(G_2 - gG_1)/h>0$, which may not be an integer.
By Lemma \ref{lem:find-lub} and Example \ref{exm:lub}, we have 
$\lub_\cE (\ep) = (0,\lub_\cE (\vEc w))$, and  Lemma \ref{lem:add-Gn}
yields
\begin{align*} 
\ep\dG <
hG_0 + \ep\dG &= hG_0+ gG_1 +  (0: \vEc w )\dG
=  G_2 +  (0: \vEc w )\dG\\
&=(0,\lub_\cE(\vEc w))\dG
=\lub_\cE(\ep)\dG. 
\end{align*}
Now, suppose that  $\ep<\delta$.
Then, $\delta=\lub_\cE^n (\ep)$ for some $n\in\nat$. 
Applying the inequality above inductively on $n$, we obtain $\ep\dG < \delta\dG$. 

Conversely, suppose that $\ep\dG < \delta \dG$.
Since $\cE$ is totally ordered,
we have either $\ep<\delta$, $\ep=\delta$, or $\delta < \ep$.
If $\ep=\delta$, then $\ep\dG =\delta \dG$, and 
if $\delta < \ep$, then our previous inductive argument implies $\delta \dG < \ep\dG$, which is also a contradiction.
Consequently, we must have $\ep < \delta$.
\end{proof}

We have the following corollary, which proves Theorem \ref{thm:main-1} (1).
\begin{cor} \label{cor:initial-unique}
If $gG_1<G_2$, then  
$G$ has \uniexp.
\end{cor}
 
\begin{deF}
\rm \label{def:cF0}
Let $\cEo:=\set{\ep \in\cE :  \ep_1=0}$, and let $\cEbar:=\set{(\ep_1) : \ep \in \cE  }$. 
Both collections contain the zero tuple $\bf 0$, and 
$\cEo \cap \cEbar=\set{\bf 0}$.
\end{deF}

The next theorem establishes the uniqueness of expansions restricted to $\cEo$.
\begin{theorem}
\label{thm:partial-unique}

Let $\set{\ep,\delta}\subset \cEo$. Then, $\ep<\delta$ if and only if $\ep\dG < \delta\dG $.  
In particular, $\ep\dG = \delta\dG $ implies that 
$\ep = \delta$.
 
\end{theorem} 

\begin{proof}
Let $Q$ be the sequence given by $Q_k=G_{k+1}$ for $k\in\nat$.
Since $gQ_1< Q_2$,  the assertion follows directly from Theorem \ref{thm:initial-unique}.
\end{proof}

\begin{cor}
\label{cor:eval}
The restriction of the evaluation map  $\eval_G$ to $\cEo$ is strictly increasing, and hence, injective.
\end{cor}

\subsection{Non-unique expansions}\label{sec:non-uniques}

We classify non-unique $\cE$-expansions 
in terms of the first few entries of tuples in $\cE$. 
\begin{deF}
\rm \label{def:tail-body}
Let $\ep\in\nat_0^\infty$. 
We define  $\ep^\circ $ to be the tuple in $\nat_0^\infty$
such that $\ep^\circ_1 =0$
and $\ep^\circ_k=\ep_k$ for all $k\ge 2$, and 
$\baR \ep:=(\ep_1)$, e.g., if $\ep=(2,3,4,5)$,
then $\ep^\circ = (0,3,4,5)$ and $\baR \ep =(2)$.
\end{deF}
Observe that  
if $\ep\in \cE$, then  $\ep^\circ\in \cEo$. 

\begin{lemma}\label{lem:split}
Let   $\ep < \delta$ in $\cE$ and
  $\ep \dG = \delta \dG$.
Then, $\ep^\circ < \delta^\circ$ in $\cE$, and 
\begin{equation}  
\baR\ep \dG - \baR\delta \dG
= \delta^\circ \dG - \ep^\circ \dG>0.
\label{eq:body-tail}
\end{equation} 
\end{lemma}
\begin{proof}
The equality part of (\ref{eq:body-tail}) follows immediately from $\ep \dG = \delta \dG$. 
The definition of $\ep<\delta$ implies that $\ep^\circ \le \delta^\circ$.
If $\ep^\circ = \delta^\circ $, then (\ref{eq:body-tail})  implies $\baR \ep= \baR \delta$, which yields a contradiction $\ep=\delta$.  Thus, we must have  $\ep^\circ < \delta^\circ$.
  By 
Theorem \ref{thm:partial-unique}, we obtain
$\delta^\circ \dG - \ep^\circ \dG > 0$, which completes  the proof.
\end{proof}

Proposition \ref{prop:cFo-diagram} below plays a crucial role
in understanding the  relationship between 
$\delta^\circ$ and $\ep^\circ$ in (\ref{eq:body-tail}).
The diagram in the proposition is generated by the sequence $\lub_{\cEo}^n(\bf 0)$ for $n\in\nat_0$, which is defined in  Definition \ref{def:imm-succ}.
\begin{notation}
\rm \label{def:imm-succ-o}
 
If $\ep\in\cEo$, then we denote $\lub_{\cEo}(\ep)$ simply by $\wt \ep$.
\end{notation}

The proof of the next lemma is analogous to that of Lemma \ref{lem:find-lub}, and is left to the reader.
\begin{lemma}\label{lem:find-lub-o} 
Let $\ep\in\cEo$, and let $n$ be the smallest integer in $\nat$ such that $ \lub_{\cE}^n(\ep)\in\cEo$.
Then, $\wt \ep = \lub_{\cE}^n(\ep)$.
\end{lemma}  

The next lemma is useful for computing $\wt \ep$.
The bijectivity of $s$ follows from the existence of its inverse map, given by 
$(\ep_1,\ep_2,\dots)\mapsto (0,\ep_1,\ep_2,\dots)$; we omit the details of the proof. 
\begin{lemma}\label{lem:s-function}
Let $s : \cEo \to \cE$ be the function given by 
$\ep \mapsto (\ep_2,\ep_3,\dots) $, e.g.,
 $\bf 0\mapsto \bf 0$.
Then, $s$ is a strictly increasing bijection, which implies that for any $\ep\in\cEo$, 
$$s\big( \lub_{\cEo}(\ep) )= \lub_{\cE}(s(\ep)).$$
In particular, we have $\wt \ep = s\Inv\big( \lub_{\cE}(s(\ep)) \big)$.
\end{lemma}

The following example demonstrates how Lemma \ref{lem:s-function} can be applied to find $\wt \ep$.

\begin{example}\rm \label{exm:lub-o}
Let $(g,h)=(2,1)$.
Let $\ep = (0,2,0,2,0,2,0,1,1)\in\cEo$.  Then, 
$ s(\ep)= (2,0,2,0,2,0,1,1)$.
From Example \ref{exm:lub}, we know 
$$\lub_\cE(s(\ep)) = (0,0,0,0,0,1,1,1).$$ 
Applying Lemma \ref{lem:s-function}, we obtain 
$$\wt \ep = s\Inv(0,0,0,0,0,1,1,1) = (0,0,0,0,0,0,1,1,1).$$

 \end{example}

\begin{deF}
\rm \label{def:ep[a,b]}
Let $\ep\in\nat^\infty$.
We define $\ep[k,\ell]:=(\ep_k,\dots,\ep_\ell)$ where $k$ and $\ell$ are integers, and 
let $\ep[k,\infty)$ denote the infinite tuple $(\ep_k,\ep_{k+1},\dots)$. 
\end{deF}
To understand the relationship between $\delta^\circ$ and $\ep^\circ$
when $\delta \dG=\ep \dG$, we consider
the  diagrams in Figures \ref{fig:g>=h}, \ref{fig:g=h-1}, and \ref{fig:g<h-1}, which represent directed weighted graphs.
The vertices correspond to the possible pairs of $\ep[2,3]$, and the transitions are formalized in Proposition \ref{prop:cFo-diagram} below.
Proposition \ref{prop:cFo-diagram} can be verified using Lemma \ref{lem:split} and Lemma \ref{lem:find-lub-o}; we omit the details. 
\begin{prop}\label{prop:cFo-diagram}
Let $\ep\in \cEo$, and let $\vEc w$ represent a tuple in $\cEp$.
Then, in Figures \ref{fig:g>=h},
\ref{fig:g=h-1}, and \ref{fig:g<h-1},   the possibilities of\ \ $\wt \ep\,[2,3]$ are represented as directed edges originating at  $\ep[2,3]$.
The resulting successor satisfies
$\wt \ep\dG = t + \ep \dG$ where $t=G_2$ if the edge has no label, and $t$ equals the edge weight otherwise. 

\begin{enumerate}
\item
Case $g\ge h$:\quad
Suppose that $\ep[2,3]=(g,h-1)$.
Then,
 $\ep =(0,g:h-1,g:\vEc w)$ if and only if\ \  $\wt \ep\, [2,3]=(0,0)$. 
 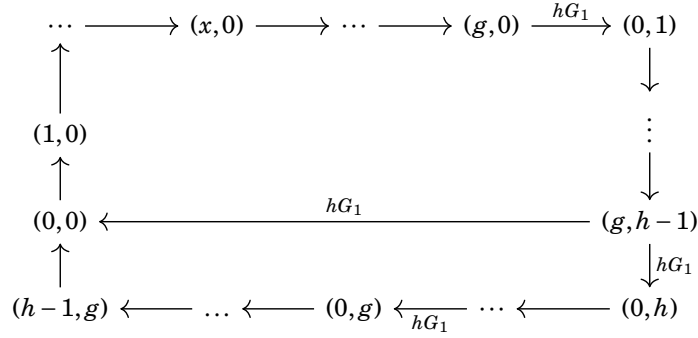
\begin{figure}[h] 
\begin{center}    \footnotesize
 \begin{tikzcd}
   \cdots \ar{r} &  (x,0)  \ar{r}  & \cdots  \ar{r}  & (g,0)  \ar{r}{hG_1}  &  (0,1) \ar d\\
     (1,0) \ar{u} &     &    &    &  \vdots \ar d\\
    (0,0)\ar{u} &    &  &  &  (g,h-1) \ar{d}{hG_1} \arrow[swap]{llll}{hG_1}\\
        (h-1,g)  \ar{u}& \dots \ar{l} & (0,g) \ar{l} & \cdots \ar{l}{hG_1} &  (0,h)  \ar{l}\\
\end{tikzcd} 
\end{center} 
\caption{
Case: $g\ge h$ }\label{fig:g>=h}
\end{figure} 

 \item
Case  $g=h-1$: 
Suppose that $\ep[2,3]=(h-1,h-1)$.
Then,
 $t=G_2$ if and only if $\ep=(0:h-1,h-1:\vEc w)$.
 \begin{figure}[h]
\begin{center}  \footnotesize
 \begin{tikzcd}
   \cdots \ar{r} &  (x,0)  \ar{r}  & \cdots  \ar{r}  & (h-1,0)  \ar{r}{hG_1}  &  (0,1) \ar d\\
     (1,0) \ar{u} &     &    &    &  \vdots \ar d\\
    (0,0)\ar{u} &    &  &  &  (h-1,h-1) \arrow[swap]{llll}{G_2\text{ or } hG_1}
\end{tikzcd}. 
\end{center} 
\caption{
Case: $g=h-1$  }\label{fig:g=h-1}
\end{figure} 

 \item
Case $g<h-1$: 
Suppose that $\ep[2,3]=(g,g)$.
Then,
$\wt \ep\, [2,3]=(0,g+1)$ if and only if    $\ep=(0,g:g,g:\vEc w)$.

\begin{figure}[h]
\begin{center}  \footnotesize
 \begin{tikzcd}
   \cdots \ar{r} &  (x,0)  \ar{r}  & \cdots  \ar{r}  &( g,0)  \ar{r}{hG_1}  &  (0,1) \ar d\\
      \vdots  \ar{u} &     &    &    &  \vdots \ar d\\
       (1,0) \ar{u} &   (h-1,g)\ar[swap]{dl}    & \cdots\ar{l}   & (g+1,g) \arrow[swap]{l}  &  (g,g)\arrow[swap]{d}{hG_1} \ar{l} \ar d\\
               (0,0) \ar{u} & (g,h-1) \ar{l}{hG_1}  
                &  \cdots\ar{l}  & (1,g+1)\ar{l}  &  (0,g+1) \ar{l}
\end{tikzcd} 
\end{center} 
\caption{Case: $g<h-1$ }\label{fig:g<h-1}
\end{figure} 
\end{enumerate}

\end{prop}

 \begin{example} \label{exm:diagram-(2,1)}
 \rm
Let $(g,h)=(2,1)$.  Then, the graph in Figure \ref{fig:g>=h}
is as follows:
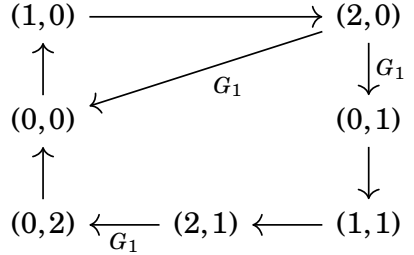
\begin{center} 
 \begin{tikzcd}
   (1,0)  \ar{rr} &      &  (2,0) \ar{d}{G_1} \ar{dll}{G_1}  \\ 
    (0,0)\ar{u} &    & (0,1) \ar{d} \\
       (0,2)  \ar{u}& (2,1) \ar{l}{G_1}  &  (1,1) \ar{l} \\
\end{tikzcd}. 
\end{center} 
Consider the case $\ep[2,3]=(2,0)$. 
Then, $\wt \ep\, [2,3]=(0,0)$ if and only if $\ep=(0,2,0,2:\vEc w)$ where $\vEc w\in \cEp$, e.g.,
  $\ep=(0,2\mid 0,2\mid 1)\implies \wt \ep=(0\mid 0\mid 0\mid 0,2)$.
\end{example}

\begin{deF}\rm \label{def:graph}
The directed weighted graph defined by the \theFigures\  
is denoted by $\Gamma_G$.
Let $\Path(\Gamma_G)$ denote the set of finite directed paths in $\Gamma_G$.
Given a path $\gamma\in \Path(\Gamma_G)$, let $\len(\gamma)\ge 0$ denote the number  of edges
in $\gamma$, and let $\wgt(\gamma)$ denote the  sum of the weights along $\gamma$.

\end{deF}
If $\len(\gamma)>0$ for $\gamma\in\thepaths$, then $\wgt(\gamma)$ belongs to the set  $\fB$ defined below:
\begin{equation}
\begin{aligned}
\fB &:= \set{mG_2 + h\ell G_1: m,\ell\in\nat_0,\ m+\ell>0},\\
\quad
\fT &:=\set{k G_1 : 1\le k\le \max\set{g,h-1}} .
\end{aligned} 
\label{eq:diff-values}
\end{equation} 
In Theorem \ref{thm:initial-values} below, 
we   relate the expressions of the elements of $\fB$ to  the RHS of (\ref{eq:body-tail}).
The expressions of the values of  the LHS of (\ref{eq:body-tail}) are listed in $\fT$ in
(\ref{eq:diff-values}) and also in Lemma \ref{lem:tail-bounds} below.
 We leave the proof of the   lemma to the reader.
\begin{lemma}
 \label{lem:tail-bounds}
If $\set{\ep,\delta}\subset \cEbar$ and $\ep\ne\delta$,
 then 
 $\abs{  \ep\dG - \delta\dG }\in \fT$,
  and in particular, we have
 $ \abs{ \ep\dG - \delta\dG } \le  G_1\max\set{g,h-1}$.
\end{lemma}

\begin{theorem}
\label{thm:initial-values}  
Let  $\ep< \delta$ in $\cE$, and let
  $\ep^*:=\ep[2,3]$  and $\delta^*:=\delta[2,3]$.
 Then, $ \ep\dG=  \delta\dG $ if and only if all of the following are satisfied:
 \begin{enumerate}
 \item  $\lub^n_{\cEo}(\ep^\circ) = \delta^\circ$ for some $n\in\nat_0$;
 \item 
There is   $\gamma\in\thepaths$ from
$\ep^* $ to $\delta^* $ such that $n=\len(\gamma)>0$;

\item 
Let $w:=\wgt(\gamma)$.
Then, $w=mG_2+h\ell G_1\in \fB$
and
$w = \delta^\circ\dG -  \ep^\circ\dG$
where  $m$ is the 
number of edges with weight $G_2$,   $\ell$ is the number of edges with weight $hG_1$, and $m=G_1p$ where $p\in\nat$;

\item \label{thm:g=f} The tail difference  $d := \bar \ep\dG - \ \bar \delta\dG$ satisfies  $d \in \fT$ 
  and 
$d=w$.
\end{enumerate} 
  
 \end{theorem}
 
 \begin{proof}

Suppose that $\ep\dG = \delta\dG$.
  Let $\gamma$ be the path in $\PathG$
  defined by $\lub^k_{\cEo}(\ep^\circ)[2,3]$
  for $0\le k\le n$ where $n$ is the unique integer
  such that  $\lub^n_{\cEo}(\ep^\circ) = \delta^\circ$.
 Lemma \ref{lem:split} implies $\ep^\circ < \delta^\circ$, and 
hence $n>0$. By Proposition \ref{prop:cFo-diagram}, 
 $\lub^k_{\cEo}(\ep^\circ)[2,3]
 \ne \lub^{k+1}_{\cEo}(\ep^\circ)[2,3]$ for all $k\ge 0$.
 So, $n=\len(\gamma)>0$.

Let $w=mG_2+h\ell G_1$ as described in the theorem.
Let us prove that $m=G_1p>0$, and leave the remainder of the proof to the reader; see  \cite{chang-2024}.
Suppose that $m=0$.
Then, $w=d$ implies that   $\ep_1-\delta_1=h\ell $.
 From \theFigures,
 it is clear that $m=0$ implies that 
 $\gamma$ consists of a single edge $(g,x)\to (0,y)$, i.e.,  $\ell=1$ and  $\ep_1-\delta_1 = h $.
 If $g\le h-1$, then  Lemma \ref{lem:tail-bounds} implies
 $w=hG_1=d\le G_1  \max\set{g,h-1}\le (h-1)G_1$, and hence, 
 $g\ge h$ must be the case.
Since $m=0$ and $\ell=1$, Figure \ref{fig:g>=h} implies that 
$\ep[2,3]=(g,a)$, i.e., 
 \GGG{
 \ep\dG = \ep_1 G_1 + g G_2 + \cdots
 =   \delta_1 G_1 +\cdots .
 }
The inequality $g\ge h$ implies that $\ep=(\ep_1,g:\vEc w)$ where $\vEc w\in\cEp$, i.e.,  $\ep[1,2] = (\ep_1,g)$ is one of the \spb s 
in the decomposition of $\ep$, and hence, $\ep_1\le h-1$.
 Since this implies that $h \ell = h= \ep_1 - \delta_1 \le h-1$, we conclude that 
 $m>0$.
 The description $m=G_1p$ where $p\in\nat$ follows from $mG_2=(\ep_1-\delta_1-h\ell)G_1$ and 
 $\gcd(G_1,G_2)=1$.
\end{proof}

As an immediate consequence, we obtain the following result from Theorem \ref{thm:main-1}.
\begin{cor} \label{cor:G2-unique}
If $G_2>\max\set{g,h-1}$, then $G$ has \uniexp.
\end{cor}

\begin{proof}
Let  $\ep$ and $\delta$ be as described in Theorem \ref{thm:initial-values}.
Suppose that $\ep\dG = \delta\dG$, and let $w=mG_2+h\ell G_1$ as described in Theorem \ref{thm:initial-values}
where $m>0$.
Then, $mG_2 = (\ep_1-\delta_1 - h\ell) G_1>0$, and $\gcd(G_1,G_2)=1$ implies
 that   $G_2$ divides $\ep_1-\delta_1 - h\ell>0$.
 Thus, $
 G_2\le  \ep_1-\delta_1 - h\ell \le\ep_1 \le \max\set{g,h-1}$. Therefore,  
  if  
$G_2>\max\set{g,h-1}$, then 
 there are no such pairs $\ep$ and $\delta$.
\end{proof}

Notice that for all three diagrams in \theFigures, 
two consecutive edges with weight $hG_1$ are in the form of 
\begin{gather} 
\overset{hG_1}{\To} (0,x) \to (1,x)\to \cdots \to (g-1,x) \to (g,x) \overset{hG_1}{\To} (0,y)
\label{eq:chain-1} \\
\intertext{or }
\overset{hG_1}{\To} (0,g) \to (1,g)\to \cdots \to  (h-1,g) \overset{ G_2}{\To} (0,0)
\to \cdots \to (g,0) \overset{hG_1}{\To} (0,y). \label{eq:chain-2}
\end{gather}
The path  in  $\PathG$ consisting of edges with weight $G_2$ described in 
(\ref{eq:chain-1}) and (\ref{eq:chain-2})
is called  a {\it chain}.

\begin{lemma}\label{lem:ell-bounds}
Let $\ep$, $\delta$, $\ep^*$,   $\delta^*$, and $\gamma$ be as defined in Theorem \ref{thm:initial-values}.
Suppose that $\ep\dG = \delta\dG$, and let $w=mG_2+h\ell G_1$ be the weight as described in the theorem 
where $m$ is a positive multiple of $G_1$. 
Then, $G_2\le \max\set{g,h-1}$, and 
\GGG{ 
    \ell \le \ell_0:=\frac{ gG_2 + G_1\max\set{g,h-1}}
  							{gG_2 + hG_1},\\
\begin{aligned}
g(\ell-1) \le& m \le \frac{G_1}{G_2}(\max\set{g,h-1}-h\ell),
\\
  					 	G_1\le &m\le g(\ell+1)+h.
\end{aligned}  
}
In particular,  if $g\ge h$, we have the following bounds: 
\begin{enumerate}
\item $\ell < g(G_2+G_1)/(gG_2 + hG_1) <g/h$ 

\item For $\ell\ge 0$, 
\GGG{
\max\set{ g(\ell-1),G_1}\le 
m\le \min\set{G_1/G_2(g-h\ell),g(\ell+1)+h}.
}
\item
 $G_1\le m\le  \min\set{g G_1/G_2,g+h} $ for $\ell=0$, 
 e.g., $g=h$.
\end{enumerate} 
If $g\le h-1$, then $\ell=0$ and $G_1\le m\le  \min\set{(h-1)G_1/G_2,g+h} $.
\end{lemma} 

\begin{proof}
If  $G_2> \max\set{g,h-1}$, then by Corollary \ref{cor:G2-unique}, 
the equality $\ep\dG = \delta\dG$ is impossible.  Thus, we have  $G_2\le \max\set{g,h-1}$.
Recall that the chains (\ref{eq:chain-1}) and (\ref{eq:chain-2}) consisting of edges of weight $G_2$ are formed by two consecutive edges of weight $hG_1$.
The numbers of edges with weight $G_2$ in these two chains
are $g$ and $g+h$, respectively.
By Theorem \ref{thm:initial-values} (3), if $\ell=0$, then $G_1\le m\le g+h$, and
$w=mG_2 \le G_1\max\set{g,h-1}$.
Since the first inequality is trivial for $\ell=0$, we proved the three 
inequalities for $\ell=0$.

Suppose that  $\ell\ge 1$.    
In this case, the path $\gamma$ must have crossed at least $\ell-1$ copies of the  chains
 (\ref{eq:chain-1}) and (\ref{eq:chain-2}).
 Thus, we have $m\ge g(\ell-1)$, and hence,
\GGG{
g(\ell-1)G_2 + h\ell G_1 \le w=mG_2 + h\ell G_1 \le G_1\max\set{g,h-1}\\
\implies
 \ell\le \frac{ gG_2 + G_1\max\set{g,h-1}}{gG_2 + hG_1},
 \quad 
 g(\ell-1) \le  m \le \frac{G_1}{G_2}(\max\set{g,h-1}-h\ell).
 }

Let us prove the last inequality.
If $g\le h-1$, then the first inequality implies $\ell=0$.
Thus, $g\ge h$ must be the case.
By the first inequality, 
 \GGG{
\ell\le 
g\,\frac{G_2 + G_1}{gG_2 + hG_1}
=\frac gh \frac{G_2 + G_1}{ G_2 +  G_1}= g/h.
 } 
 
Consider the path in Figure \ref{fig:g>=h}:
\begin{gather}
(g,0) \to (0,1) \to \cdots \to (g,h-1) \to (0,h) \to
\cdots \to  (0,g).
\label{eq:no-chain}   
\end{gather}
Thus, the number of edges with weight $hG_1$ in this path
is $  g$.  
Thus, $\ell<g/h$ implies that $\gamma$ cannot contain this path.

We claim that $\gamma$ contains at most one copy of the chain described in (\ref{eq:chain-2}),
which has   $g+h$ edges with weight $G_2$.
If it has at least two copies of the chains, then 
Figure \ref{fig:g>=h}  
implies  that $\gamma$  must contain the path (\ref{eq:no-chain}). 

Let the following be the path	 $\gamma$:
\GGG{
(x_0,y_0) \to \cdots \to  (x_k, y_k) \overset{hG_1}{\To} 
(x_{k+1},y_{k+1})\to \cdots  \to (x_n,y_n).
}
There are $\ell$ edges with weight $hG_1$, and 
in between these edges, there are $g$ edges with weight $G_2$
except for at most one pair where 
there are $g+h$ edges with weight $G_2$.
Thus, the number of edges with weight $G_2$ trapped in between 
the edges with weight $hG_1$ is at most $g(\ell-1) + h$.

Let us count the number of outer edges (with weight $G_2$).
If these outer edges do not come from the longer chain (\ref{eq:chain-2}), 
then it is clear that the number is at most $2g$.
Thus, $m\le g(\ell-1) + h + 2g=g(\ell+1) + h$.
Suppose that the left outer edges come from (\ref{eq:chain-2}).
Since the path (\ref{eq:no-chain}) cannot be contained in $\gamma$, 
Figure \ref{fig:g>=h} shows that 
 the path $\gamma$ cannot contain the longer chain (\ref{eq:chain-2}) and 
 the right outer edges cannot come from the chain (\ref{eq:chain-2}).
 So, $m\le g(\ell-1) + (g+h) + g=g(\ell+1) + h$ where the last two terms 
 come from the outer edges.
 Suppose that the right outer edges come from (\ref{eq:chain-2}).
Since the path (\ref{eq:no-chain}) cannot be contained in $\gamma$, 
Figure \ref{fig:g>=h} shows that 
 the path $\gamma$ cannot contain the longer chain (\ref{eq:chain-2}) and 
 the left outer edges cannot come from the chain (\ref{eq:chain-2}).
 For both cases, 
we have $m\le g(\ell-1) + (g+h) + g=g(\ell+1) + h$ where the last two terms 
 come from the outer edges.
Thus, $  g(\ell-1)\le m \le g(\ell+1)+h$.
This concludes the proof of the three main inequalities, and 
the remaining statements follow immediately from the three 
inequalities.
\end{proof}

Using these bounds, we prove one of the main assertions of Theorem \ref{thm:main-1}.
\begin{cor} \label{cor:G1}
The sequence $G$ has \uniexp\ if   one of the following is satisfied:
 \begin{enumerate}
 
 \item $G_2>\max\set{g,h-1}$;
 
\item $G_1>g+h$, provided that 
$g\le h-1$;
\item  $G_1>g(\ell_0+1)+h$ where $\ell_0$ is as defined in Lemma \ref{lem:ell-bounds},  provided that $g\ge h$, 
e.g., $G_1\ge (g^2+gh+h^2)/h$.

\end{enumerate}  
\end{cor}
\begin{proof} 
If $G$ does not have \uniexp,
then Lemma \ref{lem:ell-bounds} implies  $G_2\le \max\set{g,h-1}$, which proves the first result.
The second result  follows    from 
 $G_1\le m\le  \min\set{(h-1)G_1/G_2,g+h} $.
 The third result  follows    from
Lemma \ref{lem:ell-bounds} (2).
The lower bound example on $G_1$ follows from  $\ell_0<g/h$. 
\end{proof}

\begin{example} \label{exm:(81,5)}
\rm
Consider $G\in\cH$ for the $g$-golden ratio recurrence where $g=7$.
Then, by Theorem \ref{thm:g-golden}, if $(G_1,G_2)=(49,1)$, then $\#R(X) = X + O(1)$. 
 Notice that $ \ell G_1  + g\ell G_2= \ell G_3$, so 
\GGG{
 G_1 = g^2 G_2 \implies
 g G_1 = (g-1) G_1 +  g^2 G_2  
  =   (g^2-g(g-1)) G_2 +(g-1)G_3\\
  \implies
gG_1 =  gG_2 + (g-1) G_3
 }
 where the recurrence was used $g-1$ times to obtain
 the expression with $G_3$ in it.
Thus, it does not satisfy \uniexp.

If $(G_1,G_2)=(49,2)$, then $g(\ell_0+1)+1\approx 47.7$ where $\ell_0$ is the value described in 
Lemma \ref{lem:ell-bounds}.
Since $G_1>47.7$,  Corollary \ref{cor:G1} implies that $G$ has \uniexp.
Thus,  
$\#R(X) = r X + O(1) \approx 0.7495 X + O(1)$.
\end{example}
 
\begin{deF}\label{def:cFstar}
\rm
Let $\cEstar_G$ or simply $\cEstar$ denote the collection of equivalence classes  of $\cE$ 
such that the equivalence relation is given by $\ep\sim \delta$ if $ \ep\dG = \delta\dG$.
The equivalence classes are denoted by $\ecl \ep$ for $\ep\in\cE$, and if $\#\ecl \ep>1$, then
the equivalence class is called {\it non-trivial}.
\end{deF}  

\begin{theorem} \label{thm:uniform-bound}
There is  $B\in\nat$ such that $\#\ecl \ep \le B$ for all $\ep\in\cE$.
\end{theorem}

\begin{proof}
Let $\ep\in\cE$.
Since $\ep\in\ecl \ep$, by the well-orderedness,
we may choose the smallest element $\mu\in \ecl \ep$.
Thus, 
$\ecl \ep :=\set{\delta \in \cE : \mu\dG = \delta\dG,\ \mu\le \delta}$.

Let us apply Lemma \ref{lem:ell-bounds}  to $\mu<\delta$ where $\mu$ is fixed,
 and we look for all $\delta>\mu$
in $\ecl \ep$.
Since the bounds on $m$ and $\ell$ in the lemma are independent of 
 $\mu$ and $\ep$, the number of paths in $\Path(\Gamma_G)$ beginning at 
 $\mu[2,3]$ that can be considered is  bounded independently of $\mu$ and $\ep$.
 Therefore, the number of tuples in $\ecl \ep$ is bounded
 independently of $\ep$.
\end{proof}

\subsection{Classification}\label{sec:classification-sub}

Our goal in Section \ref{sec:classification} is to characterize  each equivalence class in  $\cEstar$ defined in Definition \ref{def:cFstar}.
That is, given $\ep\in\cE$, we want to be able to determine whether $
\ecl \ep $ is non-trivial. 
We show in this section
 that the non-triviality of $\ecl\ep$ is determined by the first few entries of $\ep$ (Propositions \ref{prop:V-circuit} and \ref{prop:V-chain}).

\begin{notation}
\rm
Let $\ep\in\cE$.
We   denote by $(a_1,\dots,a_n,\star)$ the tuples $\ep\in\cE$ such that 
$\ep[1,n]=(a_1,\dots,a_n)$.
When using the  colon notation for \spb\ decompositions, we may use  $\ep_1\diamond\ep_2$ to bypass the explicit block boundaries of these initial terms, e.g.,
 $\ep=(\ep_1\diamond\ep_2: a,g: \star)$ means that $(\ep_3,\ep_4)=(a,g)$ appears as a semi-proper block in the decomposition, while the specific block structures of the preceding entries are left unspecified.
\end{notation}

Consider the directed graph in Proposition \ref{prop:cFo-diagram} for $g\ge h$.  
Let $\gamma$ be a circuit in $\Path(\Gamma_G)$ that does not repeat any edges.
  If $\gamma $ contains the edge $(g,h-1)\to (0,0)$ of the weight $hG_1$,
   we call it a {\it short circuit}, and we call it a {\it long circuit} if it does not contain the edge.
   In this paper, a circuit in $\Path(\Gamma_G)$ where $g\ge h $ is
   defined to be either a long circuit or a short circuit, so that 
   it does not repeat the edges.

\begin{lemma} \label{lem:gamma-length}
Let  $\gamma$ be a path $\Path(\Gamma_G)$ of length $n$ associated with $\ep\dG = \delta\dG$, 
as defined in Theorem \ref{thm:initial-values}.
Then, $\gamma$ cannot contain a long circuit if $g\ge h$, and $\gamma$ cannot contain an edge with 
weight $hG_1$ if $g\le h-1$.
 
\end{lemma}
\begin{proof}
If $g\ge h$, then 
as noted in the proof of Lemma \ref{lem:ell-bounds}, 
the path \eqref{eq:no-chain} contains $g$ edges with weight 
$hG_1$, i.e., $\ell \ge g$.
This contradicts Lemma \ref{lem:ell-bounds} (1).
Because any long circuit must contain the path \eqref{eq:no-chain}, this contradiction completes the proof for the case $g\ge h$. 
The assertion for $g\le h-1$ follows from Lemma \ref{lem:ell-bounds} (3).
\end{proof}

The proof of the next lemma follows by tracing the paths and counting the corresponding edge transitions in Figures \ref{fig:g>=h}--\ref{fig:g<h-1}; the details are left to the reader. 

 \begin{lemma} \label{lem:long-short-circuits-0}
Let $g\ge h$.

Let $\ep=(0:0:0: \star)\in \cE$.
Then, $ \lub_{\cEo}^m(\ep)$ for $0\le m\le m_1:=g^2 + h + g$ forms a long circuit $\gamma_1$. Let $\tau:=\lub_{\cEo}^{m_1}(\ep)$. Then,
the following are true:
\begin{itemize}
\item If $\ep=(0:0:0:c:\star)$ where $c\le g-2$,
then  $\tau=(0:0:0:c+1:\star)$.
 \item If $\ep=(0:0:0:g-1:\star)$, 
then  $\tau=(0:0:0,g:\star)$.
 
  \item If $\ep=(0:0:0:c,g:\star)$ where $c\le h-2$,
then   $\tau=(0:0:0:c+1,g:\star)$.
 
   \item If $\ep=(0:0:0:h-1,g:\star)$  
then  $\tau=(0:0:0:0:\star)$.
 
\end{itemize} 

Let $\delta=(0:0:0,g: \star)\in \cE$.
Then, $ \lub_{\cEo}^m(\delta)$ for $0\le m\le m_2:=(g+1)h$ forms a short circuit $\gamma_2$. Let $\mu:=
\lub_{\cEo}^{m_2}(\ep)$. Then,
the following are true:
\begin{itemize}
\item If $\delta=(0:0:0,g:c:\star)$ where $c\le g-2$,
then  $\mu=(0:0: 0:0:c+1:\star)$.
\item If $\delta=(0:0:0,g:g-1:\star)$,
then  $\mu=(0:0: 0:0,g:\star)$.
 \item If $\delta=(0:0:0,g:c,g:\star)$ where $c\le h-2$,
then  $\mu=(0:0:0:0:c+1,g:\star)$.
  \item If $\delta=(0:0:0,g:h-1,g:\star)$,
then  $\mu=(0:0:0:0:0:\star)$.
\end{itemize} 
 
\end{lemma}

The proof of the next lemma follows from 
Lemma \ref{lem:long-short-circuits-0}, and 
we leave the proof  to the reader.

\begin{lemma} \label{lem:long-short-circuits}
Let $g\ge h$, and
let $\ep=(0:0:0:\star)$, and $\delta=(0:0:0,g:\star)$.
Then, $\lub_{\cEo}^m(\ep)$ for $0\le m\le g^2+h + g $ forms a long circuit, and 
$\lub_{\cEo}^m(\delta)$  for $0\le m\le (g+1)h+(g^2+h + g)$ forms a short circuit followed 
by a long circuit. 

If $\ep=(0:0:0:g-1:\star)$, then
$\lub_{\cEo}^m(\ep)$ for $0\le m\le 2(g^2+h + g) +(g+1)h $ forms two  long circuits and one short circuit, and 
if $\ep=(0:0:0:\star)$ where $\ep_4\ne g-1$,
then $\lub_{\cEo}^m(\ep)$ for $0\le m\le 2(g^2+h + g) $ forms two  long circuits.

\end{lemma}

Given $\ep\in \cEo$, there is the largest $\tau\in\cEo$ such that $\tau\le \ep$ and 
$\tau[2,3]=(0,0)$, and we refer to $\tau$ as the {\it circuit root of $\ep$}.
In the next lemma, we classify all circuit roots according to the \spb\  structures of  the third entry. 
Note that the \spb\  structures of  the first entry of $\ep\in\cEo$ are
given by
$(0:\star)$ and $(0,g:\star)$.
The existence of the sequences described below is guaranteed by Lemma \ref{lem:long-short-circuits}.

\begin{lemma} \label{lem:circuit-roots}Let $g\ge h$, and 
let $\vEc v$ be  finite tuples of the following forms:
\begin{equation} 
(0:0:0:\vEc u:\vEc w ),\quad  (0:0:0 ,g: \vEc w) 
\label{eq:circuit-roots}
\end{equation}
where $\vEc u$ is a \spb\ and $\vEc w \in \cEp$.

Consider the minimal finite sequence $\lub_{\cEo}^n(\vEc v)$  for $n\ge 0$ that forms
 at least two long circuits if $\vEc v=(0:0:0:\vEc u:\vEc w )$,
 and the minimal finite sequence $\lub_{\cEo}^n(\vEc v)$ for $n\ge 0$ that forms
 at least one long circuit if $\vEc v=(0:0:0 ,g:\vEc u:\vEc w) $.
 Then, the finite paths in $\Path(\Gamma_G)$ traced by these minimal sequences $\lub_{\cEo}^n(\vEc v)$ are common for all $\vEc w\in \cEp$.
 Furthermore, 
 the \spb\  structures of the first entry of each $\lub_{\cEo}^n(\vEc v)$ are common for all $\vEc w\in \cEp$ as well.
  
\end{lemma}
Lemma \ref{lem:circuit-roots} can be proved
using Lemma \ref{lem:long-short-circuits-0} and \ref{lem:long-short-circuits}. 
We leave the proof   to the reader.

\begin{deF}  \label{def:V-circuits}
\rm
If $g\ge h$, we define $V$ to be  the (finite) set of $\mathbf v\in\cE$ described in (\ref{eq:V-forms}) such that 
  there is  $ \delta\in\cE$ for which   $\vEc v\dG = \delta\dG$ and $\vEc v < \delta$:
\begin{equation} 
(a\diamond b \diamond c :\vEc u),\quad  (a\diamond b : c ,g) 
\label{eq:V-forms}
\end{equation}
where $\vEc u$ is a \spb.
\end{deF}

Since $G_k$ for $k\ge 2$ is strictly increasing, given $\vEc v\in V$, 
there are only finitely many $\delta$ described in 
Definition \ref{def:V-circuits}.
By Lemma \ref{lem:gamma-length}, 
we have an explicit upper bound on $\delta$.

\begin{prop}  \label{prop:V-circuit}
Let $g\ge h$.
Then,     $\ep\dG = \delta\dG$ where $\ep<\delta$ in $\cE$  if and only if $\ep=(\mathbf v : \star)$ 
for some $\mathbf v \in V$.
\end{prop}

\begin{proof}

Let $\ep<\delta$ in $\cE$ such that   $\ep\dG = \delta\dG$.
We
recall $\ep^\circ:=\ep[2,\infty)$.
Let $\gamma$ be the path in $\Path(\Gamma_G)$ traced by $\lub_{\cEo}^n(\ep^\circ)$ for $0\le n\le \ell$ such that 
$\lub_{\cEo}^\ell(\ep^\circ)=\delta^\circ$, so $\ell:=\len(\gamma)$.
Let $\vEc r$ be the circuit root of $\ep^\circ$, which must be of one of the forms in (\ref{eq:circuit-roots}).
Notice that 
the sequence $\lub_{\cEo}^n(\vEc r)$ for $n=1,\dots,s$ terminating at $\ep^\circ$ has length strictly less than that of a circuit.

If $\vEc  r=(0:0:0:\vEc u:\vEc w)$, then $\ep^\circ = (0\diamond b \diamond c :\vEc u:\vEc w)$, and 
if $\vEc  r=(0:0:0,g:\vEc w)$, then $\ep^\circ = (0\diamond b : c ,g:\vEc w)$.
Let $\vEc v=(\ep_1\diamond b \diamond c :\vEc u)$ or $\vEc v=(\ep_1\diamond b : c ,g)$, respectively.
It remains to show that $\vEc v \dG = \tau \dG$ for some 
$\tau\in\cE$.
By Lemma \ref{lem:gamma-length}, the path $\gamma$ cannot contain a long circuit, and 
by Lemma \ref{lem:circuit-roots}, the minimal sequence defined in the lemma is sufficiently long to contain $\gamma$.
Applying Lemma \ref{lem:circuit-roots}, we find that   for all $\vEc w'\in\cEp$, 
 the paths in $\Path(\Gamma_G)$ traced by  $\lub_{\cEo}^n(\vEc v^\circ: \vEc w')$ for $0\le n\le \ell$ are equal to $\gamma$, as are the \spb\ decomposition structures of the first entries.
 In particular, they are common if $\vEc w'=\vEc 0$.
 
 Let $\tau $ be a tuple such that $\tau_1=\delta_1$ and $\tau^\circ = \lub_{\cEo}^\ell(\vEc v^\circ)$.
 Since the \spb\ decomposition structures of the first entries of $\delta$ and $\tau$ are
 equal, we have $\tau\in\cE$.
 Then,
 \GGG{
 \vEc v_1 G_1 - \tau_1 G_1 =\ep_1 G_1 - \delta_1 G_1 
= \delta^\circ \dG - \ep^\circ \dG =\wgt(\gamma)
 = \tau^\circ \dG - \vEc v^\circ \dG .
 }
 This relation implies that $\vEc v \dG = \tau \dG$, and hence, $\vEc v \in V$.
 The proof of the converse is similar, and is left to the reader.
\end{proof}

We next consider the case $g\le h-1$.
Recall the chains (\ref{eq:chain-1}) and (\ref{eq:chain-2}). 

\begin{deF}  \label{def:chains}
\rm
The path of length $g$ consisting of edges with weight $G_2$, described in (\ref{eq:chain-1}), is called a {\it short chain} 
in  $\PathG$, and 
the path of length $h+g$ consisting of edges with weight $G_2$, described in (\ref{eq:chain-2}), is called a {\it long chain} in  $\PathG$. 
Given $\ep\in \cEo$, there is a largest \cf\ $\tau\in\cEo$ such that 
 a  chain  begins at
$\tau[2,3]$ and contains $\ep[2,3]$.
The \cf\ $\tau$ is called 
{\it the chain root of $\ep$}.
\end{deF}

We leave the proof of the next lemma to the reader.
\begin{lemma} \label{lem:long-short-chains}
Let $g\le h-1$, and
let $\ep=(0:0: \star)$ and $\delta=(0:0,g:\star)$ be chain roots.
Then, $\lub_{\cEo}^m(\ep)$ for $0\le m \le g$ forms a short chain, and 
$\lub_{\cEo}^m(\delta)$ for $0\le m\le h+g$ forms a long chain.

\end{lemma}

In the next lemma, we characterize  all chain roots in terms of the \spb\  structures of  the second entry.
The existence of the sequences described below is guaranteed by Lemma \ref{lem:long-short-chains}. We leave the details of the proof to the reader.
\begin{lemma} \label{lem:chain-roots}Let $g\le h-1$, and 
let $\vEc v$ be  finite tuples of the following forms:
\begin{equation} 
(0:0:   \vEc w ),\quad  (0:0 ,g:\vEc w) 
\notag %\label{eq:chain-roots}
\end{equation} 
where  $\vEc w \in \cEp$.

Consider the minimal finite sequence $\lub_{\cEo}^n(\vEc v)$  for $n\ge 0$ that forms
 a  chain.
 Then, the finite sequence $\lub_{\cEo}^n(\vEc v)_2$ of the second entries are common for all $\vEc w\in \cEp$, and 
 the \spb\  structures of the first entry of each $\lub_{\cEo}^n(\vEc v)$ are common for all $\vEc w\in \cEp$ as well.
  
\end{lemma}

\begin{deF} \label{def:V-chains}
\rm
If $g\le h-1$, we define $V$ to be  the (finite) set of $\mathbf v\in\cE$ described in (\ref{eq:V-forms-chain}) such that 
  there is  $ \delta\in\cE$ for which   $\vEc v\dG = \delta\dG$ and $\vEc v < \delta$:
\begin{equation} 
\vEc v=(a :  b)< (0,g),\quad  (0:0,g)\le \vEc v=(a:b ,g) \le  (g:h-1 ,g).
\label{eq:V-forms-chain}
\end{equation}   
\end{deF}
 
Observe that $\vEc v=(a,g)$ is excluded from consideration since
$\vEc v^\circ = (0,g)$ and the weight of $\vEc v^\circ \to \lub_{\cEo}(\vEc v^\circ)$ is $hG_1$.
By Lemma \ref{lem:gamma-length}, this transition cannot occur.

The next proposition provides the analogue of Proposition \ref{prop:V-circuit} for the case
$g\le h-1$.  Since its proof is similar and simpler, it is omitted here.
\begin{prop}\label{prop:V-chain}  
Let $g\le h-1$.  Then,    $\ep\dG = \delta\dG$ where $\ep<\delta$ in $\cE$  if and only if $\ep=(\mathbf v : \star)$ 
for some $\mathbf v \in V$.
\end{prop}

We   use Algorithm \ref{alg:steps} below to compute the subset $V$.
By Lemma \ref{lem:ell-bounds}, the number $\ell$ of edges with weight $hG_1$ 
in $\gamma$ described in Theorem \ref{thm:initial-values} is 
bounded, and hence, the algorithm terminates in finitely many steps.

\begin{deF}
\label{def:V-ep}\rm
Let $\ep$ be a \cf\ in $\cEo$ that is in the form of   (\ref{eq:V-forms}) or (\ref{eq:V-forms-chain}), and 
let $V_\ep$ be the collection of $\set{\vEc v, \vEc w}\subset\cE$  such that 
$\vEc v^\circ = \ep$, $\vEc v<\vEc w$, and $\vEc v\dG = \vEc w\dG$.
\end{deF}

\begin{alg}\rm \label{alg:steps}
%Let $\ep$ be a \cf\ in $\cEo$ that is in the form of   (\ref{eq:V-forms})  or (\ref{eq:V-forms-chain}).
 
The collection $V_\ep$ is computed using the following algorithm:

\begin{enumerate}

\item \label{alg:steps:n}
Let $n$ be a positive integer such that 
 the path $\gamma\in\PathG$  traced by $\lub_{\cEo}^k(\ep )$ 
for $0\le k\le n$ satisfies the inequalities on $\ell$ and $m$ in Lemma \ref{lem:ell-bounds}. 
Let  $\delta:=\lub_{\cEo}^n(\ep )$, and let $w:=\wgt(\gamma) $.
\begin{enumerate}
\item  \label{alg:steps:T}
Find two non-negative integers $t_1<s_1\le \max\set{g,h-1}$ such that $\ep':=(s_1,\ep_2,\ep_3,\dots)$ and
 $\delta':=(t_1,\delta_2,\delta_3,\dots)$ belong to $\cE$, and
 let $d:=(s_1-t_1)G_1 >0$.
\item \label{alg:steps:B}
If $d=w$, then $\ep'\dG = \delta'\dG$, and   $\set{\ep',\delta'}$ 
belongs to $ V_\ep$. 
If $d\ne w$, then $\set{\ep',\delta'}$ is not a member of $ V_\ep$.
\item
Repeat Step (\ref{alg:steps:T}) with
  a different pair of $s_1$ and $t_1$.

\end{enumerate}

 \item Repeat Step (\ref{alg:steps:n}) with a higher value of $n$.
\end{enumerate} 
By Theorem \ref{thm:initial-values},
 all members of $V_\ep$ are identified at the conclusion of the algorithm.

\end{alg}

\subsection{Example I} \label{sec:exm-I}
In this subsection, we demonstrate how to identify the set $V_\ep$ using Algorithm \ref{alg:steps} for the case $(g,h)=(2,1)$.
Lemma \ref{lem:ell-bounds} is particularly useful for this task.
Recall Definition \ref{def:graph} and the graph $\Gamma_G$ from 
   the diagram in Example \ref{exm:diagram-(2,1)}.
 By Theorem \ref{thm:initial-values} (1,2),  the tuples $\ep[2,3]$ and $\delta[2,3]$ corresponding to 
 $\ep$ and $\delta$
described in the theorem
are identified with a path $\gamma$ in $\thepaths$.
By Lemma \ref{lem:ell-bounds}, the number $\ell$ of edges with weight $hG_1$ in $\gamma$ is strictly less than $ 2$.

\begin{example}\rm \label{exm:(4,6,7)}
Let $\ep=(0:0:0:0,2)\in\cEo$, which is in the form of (\ref{eq:V-forms}), and we apply Algorithm 
\ref{alg:steps} with $n=1$.
Then, $\delta:=\lub_{\cEo}^n(\ep) = (0:1:0:0,2)$, and $w=\wgt(\gamma)= G_2 $. 

\begin{enumerate}
\item Let $s_1=1$ and $t_1=0$. Then, $d=G_1$.
 If $d=w$, then $G_1= G_2 $, which is impossible.
 Similarly, the case $s_1=2$ and $t_1=1$ cannot occur.

\item  Let $s_1=2$ and $t_1=0$.  Then, $d=2G_1$.
 If $d=w$, then $2G_1= G_2 $.
 Let $\ep'=(2:0:0:0,2)$ and $\delta'= (0:1:0:0,2)$.
 Then, they are members of $\cE$, and we can also verify it
 by checking that
 $\ep'\dG = 2G_1 +2G_5=  G_2 +2G_5=\delta'\dG$.
 Therefore,  $\set{\ep',\delta'}\in V_\ep$ if and only if $2G_1= G_2 $.
 Since $\gcd(G_1,G_2)=1$, it follows that $(G_1,G_2)=(1,2)$.

\end{enumerate} 

We next apply Algorithm 
\ref{alg:steps} with $n=2$.
This yields $\delta:=\lub_{\cEo}^n(\ep) = (0,2:0:0,2)$, and $w=\wgt(\gamma)=2G_2 $. 

\begin{enumerate}
\item Let $s_1=1$ and $t_1=0$. Then, $d=G_1$.
 If $d=w$, i.e., $G_1=2G_2 $.
 Let $\ep'=(1:0:0:0,2)$ and $\delta'= (0,2:0:0,2)$.
 Then, they belong to $\cE$, and 
 $\ep'\dG = G_1 +2G_5 = 2G_2 + 2G_5=\delta'\dG$.
 Therefore,  $\set{\ep',\delta'}\in V_\ep$ if and only if $G_1=2G_2 $.
 Since $\gcd(G_1,G_2)=1$, it follows that $(G_1,G_2)=(2,1)$.

\item  Let $s_1=2$ and $t_1=0$.  Then, $d=2G_1$.
 If $d=w$, then $2G_1=2G_2 $,  which is impossible.
 
 \item Let $s_1=2$ and $t_1=1$. Then, $d=G_1$.
 If $d=w$, then $G_1=2G_2 $.
 Let $\ep'=(2:0:0:0,2)$ and $\delta'= (1,2:0:0,2)$.
 However, we find $\delta'\not\in\cE$ in this case.

\end{enumerate} 

We apply Algorithm 
\ref{alg:steps} with $n=3$.
This yields $\delta:=\lub_{\cEo}^n(\ep) = (0:0:1:0,2)$, and $w=\wgt(\gamma)=2G_2 + G_1 $. 
\begin{enumerate}
\item Let $s_1=1$ and $t_1=0$. Then, $d=G_1$.
 If $d=w$, then $ 2G_2=0 $, which is impossible.
 Similarly, the case $s_1=2$ and $t_1=1$ does not work either.

\item  Let $s_1=2$ and $t_1=0$.  Then, $d=2G_1$.
 If $d=w$, then $ G_1=2G_2 $.
  Let $\ep'=(2:0:0:0,2)$ and $\delta'=  (0:0:1:0,2)$.
 Then, they are members of $\cE$, and 
 $\ep'\dG =2 G_1 +2G_5 =G_1 +  2G_2 + 2G_5=G_3 + 2G_5=\delta'\dG$.
 Therefore,  $\set{\ep',\delta'}\in V_\ep$ if and only if $G_1=2G_2 $.

\end{enumerate} 

For $n=4$, the only pair that works is $\ep'=(2:0:0:0,2)$ and $\delta'=  (0:1:1:0,2)$, which requires $G_1=3$ and $G_2=1$.
For $n=5$, the only pair that works is $\ep'=(2:0:0:0,2)$ and $\delta'=  (0,2:1:0,2)$, and 
which requires $G_1=4$ and $G_2=1$.
For $n\ge 6$, the path $\gamma$ contains two edges with weight $hG_2$, so this concludes
our identification of $V_\ep$.
By Proposition \ref{prop:V-circuit},
  if
$\tau=(\vEc v:\star)\in\cE$ where $\set{\vEc v,\vEc w}\in V_\ep $ where $\vEc v < \vEc w$, then there is $\delta >\tau$ in $\cE$ such that $\tau\dG = \delta\dG$.

\end{example}

 \begin{example}\rm \label{exm:short-circuit}
Let $\ep=(0:0:0,2)\in\cEo$.  Then, the algorithm terminates after checking $1\le n\le 5$, and
$V_\ep$ contains the following five pairs:

\begin{enumerate}
\item $n=1$: $\ep'=(2:0:0,2:\star) \to \delta'=(0:1:0,2:\star),\  G_2=2G_1$
\item $n=2$: $\ep'=(1:0:0,2:\star) \to \delta'=(0,2:0,2:\star),\  2G_2= G_1$
\item $n=3$: $\ep'=(2:0:0,2:\star) \to \delta'=(0:0:0: \star ),\  2G_2= G_1$
\item $n=4$: $\ep'=(2:0:0,2:\star) \to \delta'=(0:1:0: \star ),\  3G_2= G_1$
\item $n=5$: $\ep'=(2:0:0,2:\star) \to \delta'=(0,2:0: \star ),\  4G_2= G_1$

\end{enumerate}

\end{example}

Appearing in  \appone\ is the list of all pairs $\ep'$ and $\delta'$ 
along with the equality $d=w$  that are obtained by
using Algorithm \ref{alg:steps} and the analysis demonstrated in 
Example \ref{exm:(4,6,7)}.  The column Ratio  in the table will be explained in 
Section \ref{sec:second-term}, and we summarize the discussion as follows:
\begin{prop} \label{prop:list-A}
Let $\ep\in\cE$.
Then, $\ep $ is equal to  one of the tuples appearing in \appone\ if and only if 
there are $\ep<\delta$ in $\cE$ such that $\ep\dG=\delta\dG$.
\end{prop}

Recall the finite set $V$ from Definition \ref{def:V-circuits}.
In \appone, we further reduced the expressions of the tuples in $V$ 
if they have the same expressions of $w$ and $d$.
For the case of $2G_1=G_2$, all pairs $\ep<\delta$ in $\cE$ with $\ep\dG = \delta\dG$ come from
$w=G_2$ and $d=2G_1$, and $\ep$ can be further reduced to $(2:\star)$ from $(2:0:\star)$,
$(2:1:\star)$, and $(2:0,2:\star)$.

\subsection{Example II} \label{sec:exm-II}
Let $(g,h)=(1,2)$.
Then, we have the following graph:
 
\begin{center} 
 \begin{tikzcd}
   (1,0)  \ar{rr}{2G_1} &      &  (0,1) \ar{d}   \\ 
    (0,0)\ar{u} &    & (1,1). \ar{ll}{G_2,\ 2G_1} 
\end{tikzcd} 
\end{center} 
Listed below are  all pairs $\ep'$ and $\delta'$ 
along with the equality $d=w$  that are obtained by
using Algorithm \ref{alg:steps}:
\begin{equation*}
\begin{array}{| c  | c | c | l | l | l |}
\hline
\text{Equality} &w & d & \hfill\ep'\hfill & \hfill\delta' \hfill & \hfill\text{Ratio}\hfill \\
 \hline
\VS{1.2em}
 \mrow{1}{ G_1= 2G_2}
  	& \mrow{1}{2G_2} & \mrow{1}{ G_1} &  (1:0,1:\star) & (0:0: \star) & vr \ome^3\\   
       &   &  &   (1:1,1:\star) & (0:1: \star) &  vr \ome^3 \\   
\hline
\VS{1.2em}
\mrow1{ G_1=3G_2} & \mrow1{3G_2} & \mrow1{G_1} &  (1:0,1:\star) & (0:1: \star) & vr \ome^3\\  
 \hline 
\end{array} 
\end{equation*}

 \subsection{Example III} \label{sec:exm-III}
Let $(g,h)=(1,3)$.
Then, we have the following graph:
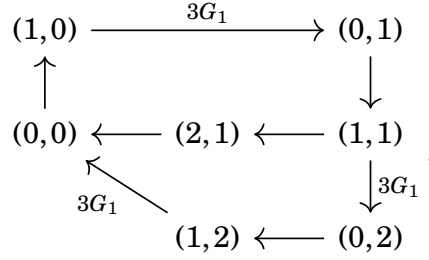
\begin{center} 
 \begin{tikzcd}
   (1,0)  \ar{rr}{3G_1} &      &  (0,1) \ar{d}     \\ 
    (0,0)\ar{u} &  (2,1) \ar{l}  & (1,1) \ar{d}{3G_1}  \ar{l}\\
       & (1,2) \ar{ul}{3G_1}   &  (0,2) \ar{l} \\
\end{tikzcd}. 
\end{center} 
Listed below are
 all pairs $\ep'$ and $\delta'$ 
along with the equality $d=w$  that are obtained by
using Algorithm \ref{alg:steps}:
\begin{equation*}
\begin{array}{| c  | c | c | l | l | l |}
\hline
\text{Equality} & w & d & \hfill\ep'\hfill & \hfill\delta' \hfill & \hfill\text{Ratio}\hfill \\
 \hline
\VS{1.2em}
 \mrow{1}{ G_1= 2G_2}
  	& \mrow{1}{ 2G_2} & \mrow{1}{  G_1} &  (1:0,1:\star) & (0:2,1:\star) & vr \ome^3\\   
       &   &  &   (1:1,1:\star) & (0:0: \star) &  vr \ome^3 \\   
            &   &  &   (1:2,1:\star) & (0,1: \star) &  vr \ome^3 \\   
\hline
\VS{1.2em}
\mrow1{ G_1=3G_2} & \mrow1{3G_2} & \mrow1{G_1} &  (1:0,1:\star) & (0:0: \star) & vr \ome^3\\  
 &   &  &  (1:1,1:\star) & (0:1: \star) & vr \ome^3\\  
\hline
\VS{1.2em}
\mrow1{ G_1=4G_2} & \mrow1{4G_2} & \mrow1{G_1} &  (1:0,1:\star) & (0:1: \star) & vr \ome^3\\   
 \hline 
\end{array} 
\end{equation*}

\section{Calculation of $\#R(X)$}\label{sec:R(X)}

Throughout the section, let $G$ be a sequence in $\cH$ defined in Definition \ref{def:3-recurrence}, and 
let $X$ denote a large positive integer. 
 Recall the subset $R (X)$ of integers from Definition \ref{def:R} and 
 the set $\cEstar_G$ of equivalence classes from Definition \ref{def:cFstar}.
 \begin{deF}
\rm
Let $\cEstar_G(X):=\set{\ecl \ep \in \cEstar_G: \ep\dG\le X}$.  
When the context is clear, we simply 
write $\cEstar(X)$.
\end{deF}

The goal of our paper is to find an asymptotic formula of $\#R(X)$, and 
we have 
\begin{gather}
\#R(X) = \#\set{ \ep\in \cE : \ep\dG\le X} -  
\sum_{\ecl \ep\in\cEstar(X)}\big( \#\ecl \ep-1\big) .\label{eq:R(X)}
\end{gather}  
 Recall the subset $V$ from Definition \ref{def:V-circuits} 
  and \ref{def:V-chains}, and define
 \begin{equation}
V^\star := \set{ (\vEc v : \vEc w) \in\cE: \vEc v \in V,\ \vEc w \in\cEp}.\label{eq:Vstar}
\end{equation}
Notice that 
if $\ecl \ep\in \cEstar$  and there is  the largest tuple  $\delta$   in $\ecl \ep$ not equal to $\ep$, then Propositions \ref{prop:V-circuit} and \ref{prop:V-chain} imply $\delta \not \in V^\star$  and
$V^\star\cap \ecl \ep = \ecl \ep\backslash\set{\delta}$.
Thus, $\#\ecl \ep-1 = \#(V^\star\cap \ecl \ep)$, and  
\begin{gather}
\sum_{\ecl \ep\in\cEstar(X)}\big( \#\ecl \ep-1\big)
=\#\set{ \ep \in V^\star : \ep\dG \le X}.
\label{eq:v:w}
\end{gather}  
In Section \ref{sec:second-term},  we introduce an 
asymptotic formula for  the number of $\ep\in\cE$
such that  $\ep\dG\le X$ and $\ep=(\vEc v: \vEc w)$ where $\vEc w$ varies in $ \cEp$ and $\vEc v\in V$ is fixed (Proposition
\ref{prop:S(v,X)}).
By (\ref{eq:Vstar}) and 
(\ref{eq:v:w}), the formula will yield an answer for the second term of (\ref{eq:R(X)}).

\subsection{The first term}

We estimate the first term of (\ref{eq:R(X)}).  Recall the \funds\ $H$ of $\cE$ from Definition \ref{def:expansions-eval},
and its description $H_k=a_0 \al^k + b_0 \beta^k$ from Definition \ref{def:char-poly}.
Let $a_1$ and $b_1$ be two real numbers such that  
\begin{gather}
\label{eq:G-binet}
G_k = a_1\al^k + b_1 \beta^k,\ \forall k\in\nat.\\
\intertext{Using the initial values of the sequences, we find}
a_0=\frac{(g+1) \al + h}{\al^2(\al-\beta)},\quad
a_1=\frac{G_2\al + hG_1}{\al^2(\al-\beta)}.
\end{gather}

\begin{deF}
\label{def:S(X)}\rm
Define $S(X):=\set{ \ep\in \cE : \ep\dG\le X}$, and $T(X):=\set{ \ep\in \cE : \ep\dH\le X}$.
\end{deF}

We   estimate $\#S(X)$ using $\#T(X)$.
Let us prepare some lemmas 
 for error terms.
Let $M:=\max\set{g,h-1}$, and let $\beta_0:=\abs{\beta}$.
Recall the error term $E(X)$ from (\ref{eq:E(X)}).
\begin{lemma}\label{lem:m}
Let $X\in \nat$.
Let $\ep\in S(X)$, and 
let $m:=\ord(\ep)$.
Then, 
$$m \le \log_\al (X)  + O(1),$$
 and
$\sum_{k=1}^m \beta_0^k =O(E(X))$.
\end{lemma} 
\begin{proof}

Since $\ep_k\le M$, we have 
\GGG{
 \ep\dG\le M \sum_{k=1}^m G_k=M \sum_{k=1}^m a_1\al^k + b_1\beta_0^k
 =M \sum_{k=1}^m \al^k\left(a_1 + b_1(\beta_0/\al)^k\right)\\
\implies 
\ep\dG=O(\al^m)   \le X 
\implies
\al^m O(1)\le X\\
\implies
m \le \log_\al (X)  + O(1).
}
If $\beta_0<1$, then $\sum_{k=1}^m \beta_0^k=O(1)$.
If $\beta_0=1$, then $\sum_{k=1}^m \beta_0^k=m=O(\ln X)$.
If $\beta_0>1$, then there is a positive number $C$ independent of $m$ and $X$ such that 
\GGG{
\sum_{k=1}^m \beta_0^k<C\beta_0^m \le C\beta_0^{\log_\al X+O(1)}
=C\beta_0^{O(1)} X^{\log_\al \beta_0}.
}
\end{proof}

Note that 
\begin{gather} \label{eq:Gk-Hk}
G_k = \tfrac{a_1}{a_0} H_k + O(\beta_0^k), \quad
\frac{a_0}{a_1} = \frac{(g+1)\al + h}{ G_2\al + hG_1}=r  
\end{gather} 
where $r$ is defined in Definition \ref{def:R}.

 \begin{theorem} \label{thm:S(X)}
For $X\in\nat$, we have $\#S(X) = r X + O(E(X))$.
\end{theorem}

\begin{proof}
By (\ref{eq:Gk-Hk}),
there is   $C_1\in \nat$  such that 
 \GGG{
 \tfrac{a_1}{a_0} H_k -C_1\beta_0^k < G_k <\tfrac{a_1}{a_0} H_k + C_1\beta_0^k,\\
 \tfrac{a_0}{a_1} G_k -C_1\beta_0^k < H_k <\tfrac{a_0}{a_1} G_k + C_1\beta_0^k 
 }
  for all $k\in\nat$.
  By Lemma \ref{lem:m}, there is $C_0\in\nat$  such that
$-C_0E(X) < \sum_{k=1}^m\beta_0^k<C_0 E(X)$.

Let us show that there is an injection from $S(X)$ to $T(\frac{a_1}{a_0} X+E'(X))$ given by $\ep\mapsto \ep$
where $E'(X)=C_1MC_0\frac{a_0}{a_1}E(X)$.
Let $\ep\in S(X)$, and  $m:=\ord(\ep)$.  Then,
\GGG{
 \ep\dG = \sum_{k=1}^m \ep_k G_k > 
 \sum_{k=1}^m \ep_k  \tfrac{a_1}{a_0} H_k - C_1M\beta_0^k 
 =\tfrac{a_1}{a_0} \ep\dH - C_1M\sum_{k=1}^m\beta_0^k  \\
  > \tfrac{a_1}{a_0} \ep\dH  - C_1M C_0 E(X).\\
 \intertext{Then, $\ep\dG \le X$ implies that }
  \tfrac{a_1}{a_0} \ep\dH  - C_1M C_0 E(X) < X
  \implies
  \ep\dH <\tfrac{a_0}{a_1} X+ C_1MC_0\tfrac{a_0}{a_1}E(X).
}
Thus, $\ep\in T(\frac{a_0}{a_1} X+E'(X))$.

Similarly, there is an injection from $T(\frac{a_0}{a_1} X-E''(X))$ to $S(X)$ given by $\ep\mapsto \ep$
where $E''(X)=C_1MC_0 E(X)$.
Let $\ep\in T(\frac{a_0}{a_1} X-E''(X))$.
Then,
\GGG{
\ep\dH> \sum_{k=1}^m \ep_k  \tfrac{a_0}{a_1} G_k - C_1M\beta_0^k
>\tfrac{a_0}{a_1}\ep\dG -C_1MC_0 E(X).\\
\ep\dH\le \tfrac{a_0}{a_1} X-E''(X) \implies
\tfrac{a_0}{a_1}\ep\dG -C_1MC_0 E(X)< 
\tfrac{a_0}{a_1} X-E''(X).
}
This implies $\ep\dG < X$, and therefore $\ep\in S(X)$.

 Injectivity implies that 
$\#T(\frac{a_0}{a_1} X-E''(X))\le \#S(X) \le \#T(\frac{a_0}{a_1} X+E'(X))$, and 
since $H$ is a \funds\ of $\cE$, we have 
\GGG{
\tfrac{a_0}{a_1} X-E''(X) \le \#S(X)\le \tfrac{a_0}{a_1} X+E'(X),
}
which proves the assertion.
\end{proof}

 \begin{cor}
If $G\in\cH$ satisfies \uniexp, then $\#R(X)=rX + O(E(X))$.
\end{cor}

\begin{proof}
If $G\in\cH$ satisfies \uniexp, then $\#[\ep]=1$ for all $\ep\in\cE$.
By (\ref{eq:R(X)}), we obtain $\#R(X) = \#S(X)$.
\end{proof}

 \subsection{The second term}\label{sec:second-term}
 
 In this section, we introduce the formula for the description (\ref{eq:v:w}) of the second term of $\#R(X)$, and 
 at the end of the section, we prove Theorem \ref{thm:main-1}. 
 Recall the subcollections $\cEp$ and $\cEm$ from Definition \ref{def:Z-collection}.

\begin{deF}
\label{def:Tp}\rm
Let $\Tp(X):=\set{\ep\in T(X) : \ep\in\cEp}$, $\Tm(X):=\set{\ep\in T(X) : \ep\in\cEm}$, $\Sp(X):=\set{\ep\in S(X) : \ep\in\cEp}$,
and
$\Sm(X):=\set{\ep\in S(X) : \ep\in\cEm}$.
\end{deF}

\begin{prop} \label{prop:Un}
Let  $U_n:=\#\Tp(H_{n+1}-1)$ for $n\in\nat$. 
Then, $U_n$ is equal to the number tuples $\ep$ in  $\cEp$  
such that $\ord(\ep)\le n$, and 
\GGG{
  U_{n+2}=g U_{n+1} + h U_n
  }
  for all $n\in\nat$ and $(U_1,U_2)=(g,g^2+h)$.
\end{prop}

\begin{proof}
Let $n\ge 3$, and 
$\ep \in \Tp(H_{n+1}-1)$.  
Notice that
$$H_{n+1}-1= (\dots,h-1,g: h-1,g)\dH.$$
Since $H$ is a \funds, we have $\ord(\ep)\le n$.
Consider the \spb\ decomposition structure of the first entry of $\ep$,
i.e., $\ep=(a:\vEc w)$ or $\ep=(b,g:\vEc w)$ where $0\le a\le g-1$ and 
$0\le b\le h-1$.
By definition of $U_n$, the number of $\vEc w$ for the first case is $U_{n-1}$ for each $a$, and 
the number for the second case is $U_{n-2}$ for each $b$.
So, $U_n= g U_{n-1} + h U_{n-2}$ for $n\ge 3$.
It is clear that $U_1=g$.  For the value of $U_2$, consider $(a_1:a_2)$ where $0\le a_k \le g-1$ and 
$(b, g)$ where $0\le b\le h-1$.  Hence, $U_2 = g^2 + h$.
\end{proof}

By (\ref{eq:Gk-Hk}), as $n\to\infty$, we have
\begin{gather} 
U_n/H_{n+1} \to \frac{ (g^2+h)\al +gh}{\al( (g+1)\al + h ) }
=\frac{ g\al +h}{ (g+1)\al + h  }=\frac{\al}{\al+1}=\frac1{1+\ome} 
\end{gather}
where we used $\al^2=g\al +h$.

\begin{deF}\rm
Let $v:= 1/(1+\ome)$.
\end{deF}

\begin{cor}\label{cor:un}
For $n\in\nat$, we have $ U_n - v H_{n+1} = O(\beta_0^{n})$.
If $X=H_{n+1}-1$, then $\#\Tp(X) = v X +  O(\beta_0^{n})$.
\end{cor}

\begin{theorem} \label{thm:Tp(X)}
For $X\in\nat$,
  $\#\Tp(X) = v X + O(\sum_{k=1}^n \beta_0^k)$
  if $H_n \le X < H_{n+1}$.
In particular,  $\#\Tp(X) = v X + O(E(X))$.
\end{theorem}

\begin{proof}
 We use induction on $n$.
  By choosing the constant of the big-$O$ notation sufficiently large,
  we see that the statement is true for all $n\le 3$.
 Suppose that there is $n\ge 3$ such that the statement is true for all $m  <n$.
 
 Consider the case $X=\ep\dH$ such that $\ep=(\vEc w: a)$ where 
 $\vEc w \in \cEp$ and $\ord(\ep)=n$.  Let $x:=\delta\dH\le X$ be an integer
 where $\delta=(\vEc w' : b)$, $\vEc w'\in\cEp$, $\ord(\vEc w')\le n-1$, and $0\le b<a$.
 Then, by Proposition \ref{prop:Un},  the number of such $x$ is $aU_{n-1}$.
Let $y:=\mu\dH\le X$ be an integer
 where $\mu=(\vEc w' : a)$ and $\vEc w'\in\cEp$.
  Then, $aH_n \le y \le X$, and $0\le y - a H_n \le X- aH_n$.
By the induction hypothesis,   the number of such $y$ is 
$v(X-a H_n) +  O(\sum_{k=1}^{n-1} \beta_0^k) $.
Thus, the total number is 
\GGG{
aU_{n-1}  + v(X-a H_n) +  O\left(\sum_{k=1}^{n-1} \beta_0^k\right) 
=vX + a( U_{n-1} -v H_n )
	+  O\left(\sum_{k=1}^{n-1} \beta_0^k\right).\\
	\intertext{By Corollary \ref{cor:un},} 
	\#\Tp(X) = vX +O(\beta_0^{n-1}) 
+   O\left(\sum_{k=1}^{n-1} \beta_0^k\right)
	= vX  +   O\left(\sum_{k=1}^{n } \beta_0^k\right).
}

 Consider the case $X=\ep\dH$ such that $\ep=(\vEc w: a,g)$ where 
 $\vEc w \in \cEp$ and $\ord(\ep)=n$. 
 The positive integers that are less than or equal to $ X$ are given by
 $\delta\dH$ where $\delta=(\vEc w':b)\in\cEp$  and $0\le b<g $,
 $\mu\dH$ where
  $\mu=(\vEc w':b,g)$, $\ord(\mu)=n$, and $0\le b<a$,
  and integers $y$ such that $aH_{n-1} + gH_n \le y\le X$.
By the induction hypothesis, we obtain
\AAA{
  \#\Tp(X) &= g U_{n-1} + a U_{n-2} + v(X-aH_{n-1} - gH_n) 
  +O\left(\sum_{k=1}^{n-2 } \beta_0^k\right)\\
 & =vX +O\left(\sum_{k=1}^{n } \beta_0^k\right).
  }

\end{proof}

The proof of the next corollary is similar to that of Theorem \ref{thm:S(X)} where 
  we constructed injections around $S(X)$, and we leave it to the reader:
\begin{cor} \label{cor:Sp}
For $n\in\nat$, we have 
$\#\Sp(X) =v r  X + O(\sum_{k=1}^n\beta_0^n)$
if $H_n\le X<H_{n+1}$.
In particular, $ \#\Sp(X) =v r  X + O(E(X))$.
\end{cor}

 \begin{deF}\rm \label{def:Gm}
Given $m\in\nat$, let $\Sp^m(X)$ 
be the subset of $\Sp(X)$ consisting of $
\ep $ such that 
$\ep_k=0$ for all $1\le k\le m-1$ and
$(\ep_m,\ep_{m+1},\dots)\in \cEp$, e.g.,
$(0,0,g)\not\in \Sp^3(X)$ for any $X$, but
$(0,0,0,g)\in \Sp^3(X)$ for large $X$.

\end{deF}

The proof of the next corollary can be obtained 
from the following observation  
and Corollary \ref{cor:Sp}, and we leave it 
to the reader.
If $\ep\in \Sp^m(X)$ and $\vEc w=(\ep_m,\ep_{m+1},\dots)\in \cEp$,
then
\AAA{
 \ep\dG &= \sum_{k=m}^\infty \ep_k G_k=
 \sum_{k=m}^\infty \ep_{k} a_1 \al^k+ O(E(X))\\
&=\al^{m-1}\sum_{k=m}^\infty\ep_{k} a_1 \al^{k-m+1}+ O(E(X))
=
\al^{m-1} (\vEc w\dG) + O(E(X)).
} 
 
\begin{cor} \label{cor:Gm}

For $m\in\nat$, we have 
$\#\Sp^m(X) =  vr\ome^{m-1} X + O(E(X))$. 
\end{cor}

\begin{deF}
\label{def:Rv}\rm
Let $\vEc v\in\cE$.
Define 
$$S(\vEc v:X):=\set{\ep=(\vEc v: \vEc w) : \vEc w\in\cEp,\ \ep\dG\le X}.$$
\end{deF}

Proposition \ref{prop:S(v,X)} follows immediately from Corollary \ref{cor:Gm}.
\begin{prop} \label{prop:S(v,X)}
Let $\vEc v\in\cE$ such that $m:=\ord(\vEc v)\ge 1$.
Then, $$\#S(\vEc v:X) = vr\ome^{m} X + O(E(X)).$$
\end{prop}

The last columns of the tables in \appone\ and Examples \ref{sec:exm-II} and \ref{sec:exm-III} list the limiting ratios of the number of $\ep'\in \cE$ in the given form such that $\ep'\dG\le X$.  They are given by the formula in Proposition \ref{prop:S(v,X)}.

\begin{proof}[Proof of Theorem \ref{thm:main-1}]
We use Proposition \ref{prop:S(v,X)}.
By (\ref{eq:R(X)}), (\ref{eq:Vstar}),  and (\ref{eq:v:w}), we have 
\begin{align}
\#R(X) &= rX - \sum_{\vEc v \in V} \#S(\vEc v : X) + O(E(X))\notag\\
&= rX\left( 1- v \sum_{\vEc v \in V} \ome^{\ord(\vEc v)} \right)+ O(E(X)).\notag\\
&=\frac{r}{1+\ome}X\left(1+\ome-  \sum_{\vEc v \in V} \ome^{\ord(\vEc v)} \right)+ O(E(X)).
\label{eq:R(X)-2}
\end{align} 
Since $\ord(\vEc v)\le 5$ for all $\vEc v\in V$, this 
completes the proof of the first statement.
\end{proof}

\subsection{The examples} \label{sec:the-exm}  

In this section, we  revisit the examples considered in Sections \ref{sec:exm-I}, \ref{sec:exm-II},
and \ref{sec:exm-III}, and estimate the value of $\#R_G(X)$.
The following identity will be used to simplify expressions in terms of $\ome^k$:
\GGG{
\ome^n = g \ome^{n+1} + h  \ome^{n+2},\ \forall n\in\zz.
}
 
\begin{example}\rm \label{exm:(2,1)}
Consider the case $(g,h)=(2,1)$.  Then, \appone\ shows all tuples $\vEc v$ in reduced expressions that 
we need to consider for the sum described in (\ref{eq:R(X)-2}).
Shown
under the column Ratio in \appone\ are the limiting ratios of $
  \#S(\vEc v: X)/X$, which is $vr\ome^{\ord(\vEc v)}$.
  
For example, if $G_1=2G_2$, i.e., $(G_1,G_2)=(2,1)$, then the table lists five tuples
$\vEc v$, which are $(1,0)$, $(2,0)$, $(2,1)$, $(1,0,2)$, and $(2,0,2)$.
Thus,
    \GGG{
    \sum_{\vEc v \in V} \ome^{\ord(\vEc v)}
    = 3\ome^2 + 2\ome^3=(2\ome^2 +  \ome^3)+ \ome^2 +  \ome^3\\
    =\ome+ \ome^2 +  \ome^3.
    }
This expression is found in the far right column of \appone.
The three forms $(2:0:\star)$, $(2:1:\star)$, and $(2:0,2:\star)$ can be reduced to $(2:\star)$, 
but as their associated weights $w$ are different, they were not combined to $(2:\star)$.

Therefore, if $G_1=2G_2$, i.e., $(G_1,G_2)=(2,1)$, then
\begin{gather}
1+\ome-  \sum_{\vEc v \in V} \ome^{\ord(\vEc v)}=
1+ \ome-(\ome+ \ome^2 +  \ome^3)\HSW{.3} \notag\\
=1- \ome^2 - \ome^3=2\ome - \ome^3=\ome+2\ome^2 =\ome(1+2\ome) , \notag\\
\frac{r }{1+\ome} \left( 1+\ome-  \sum_{\vEc v \in V} \ome^{\ord(\vEc v)}\right)
=\frac{r\ome(1+2\ome)}{1+\ome}=\frac{3 +\ome}{1+ 2\ome}\frac{\ome(1+2\ome)}{1+\ome}=1
  \label{eq:=1}\\
\implies \#R(X)= X + O(1). \notag
\end{gather}
Computer calculation suggests that $\#R(X) = X$ for all $X\in\nat$.

Listed below are the ratios predicted by  Theorem \ref{thm:main-2}, Part (1) and the actual ratios for $X=10^4$:
\begin{equation*}
\begin{array}{| c | c | c |}
\hline\VS{1.1em}
\text{Equality} & \text{Formula ratio} & \text{Actual ratio} \\
\hline
2G_1 = G_2 & 1 & 1 \\
\hline
 G_1 =2 G_2 & 1 & 1 \\
\hline
 G_1 =3 G_2 & 1 & 1 \\
\hline
 2G_1=3G_2& 1 & .9999 \\
\hline
 G_1 =4 G_2 & 1 & .9999  \\
\hline
  G_1 =5 G_2 & 1 & .9998  \\
\hline
\end{array} .
\end{equation*}
According to the computer calculations, for the case $2G_1=3G_2$, $R(X)$ is missing $1$,
for the case $ G_1=4G_2$, $R(X)$ is missing $3$, and for the case $ G_1=5G_2$, $R(X)$ is missing $3$ and $4$.

In Section \ref{sec:g-golden}, we use a different approach to 
prove that the formula simplifies to $1$ in this case. 
However, the way the right-hand side of the formula in Theorem \ref{thm:main-1} 
simplifies to $1$ in \eqref{eq:=1} strikes us as a combinatorial miracle.
\end{example}

\begin{example}\rm \label{exm:(1,2)}
Consider the case $(g,h)=(1,2)$.  Then, the table in Section \ref{sec:exm-II} shows all tuples $\vEc v$ in reduced expressions that 
we need to consider for the sum described in    (\ref{eq:R(X)-2}).
By repeating calculations as in Example \ref{exm:(2,1)}, we obtain   the ratios
 predicted by  Theorem \ref{thm:main-2} (2) and the actual ratios for $X=10^4$:

\begin{equation*}
\begin{array}{| c | c | c |}
\hline\VS{1.1em}
\text{Equality} & \text{Formula ratio} & \text{Actual ratio} \\
\hline 
 G_1 =2 G_2 & .8333 & .8333 \\
 \hline
 G_1 =3 G_2 & .6875 & .6873 \\
\hline
\end{array} .
\end{equation*} 

For the case $(g,h)=(1,3)$, we have the following: 

\begin{equation*}
\begin{array}{| c | c | c |}
\hline\VS{1.1em}
\text{Equality} & \text{Formula ratio} & \text{Actual ratio} \\
\hline 
  G_1 =2 G_2 & .7591 & .7596 \\
 \hline
 G_1 =3 G_2 & .5961 & .5953 \\
 \hline
 G_1 =4 G_2 & .5014 & .5004 \\
\hline
\end{array} .
\end{equation*} 
The error term  for this case is $O(X^{\log_\al\beta_0})= O(X^{0.3171})$, 
and compared to the other cases, higher differences between the ratios are observed.
\end{example}

 \section{The $g$-golden ratio recurrence} \label{sec:g-golden}
 
 We prove Theorem \ref{thm:g-golden} in this section.
Notice that  Theorem \ref{thm:main-2} (1) proves Theorem \ref{thm:g-golden} for $g=2$. Thus,
 we prove the case $g\ge 3$ and $h=1$.
 
 Throughout this section, let $ \vEc w$ and $\vEc w^k $ denote  tuples in $\cEp$ where $k\in\nat_0$.
Using the notation $\vEc w^{k+1}$, we indicate that it is a tuple potentially 
 different from $\vEc w^k$;
 especially, $\vEc w^k$ does not denote the $k$th power in any way.
 For example, if $0\le   \vEc w_1<g-1$, then
  we may write $G_1 + \vEc w\dG = \vEc w^1 \dG$, i.e., 
 if $\vEc w = (a:\star)$ where $a<g-1$, then $\vEc w^1 = (a+1: \star)$, and 
 if $\vEc w = (0,g:\star)$, then $\vEc w^1 = (0: \star)$, which includes both possibilities
$\vEc w^1 = (0: 0:\star)$ and $\vEc w^1 = (0: 0,g: \star)$.

For any integers $a_1,\dots,a_n$, we define 
\GGG{
[ a_1,\dots,a_n ]:=
\sum_{k=1}^n a_k G_k.
}
If a colon is used instead of a comma in between
$a_k$ and $a_{k+1}$,
 it indicates that $(a_{k+1},\dots, a_n)\in \cEp$, e.g.,
 \GGG{
 [g,g-1,g-2,1]= [g:g-1,g-2,1]=
 [g:g-1:g-2,1].
 }
 If $\vEc w=(\vEc w_1,\dots, \vEc w_m)\in \cEp$, then we also define
 \GGG{
 [a_1,\dots,a_n : \vEc w]:=
  [a_1,\dots,a_n, \vEc w_1,\dots, \vEc w_m]\in\zz.
  }
  
The strategy for proving  the theorem is to show that given a sufficiently large $\delta\in\cE$, 
there is $\ep\in\cE$ such that 
\begin{equation}
1+\delta\dG = \ep \dG.
\label{eq:1+delta}
\end{equation}
Once established, this relation implies by induction that 
$\#R_G(X) = X + O(1)$.

\subsection{Case I}
We consider the case
 $(G_1,G_2)=(q,1)$ in this section.
Our analysis depends on the decomposition structure of the first entry of $\delta$, as described in \eqref{eq:1+delta}.
 Lemma \ref{lem:add-Gn} implies
\begin{equation} 
\begin{gathered}
  \vEc w_1\ne g-1 
\implies
1+[a: \vEc w] =G_2+[a: \vEc w]
=[a: \lub_\cE(\vEc w)]\in R_G,\\ 
0<a\le g\implies 
1+[a:g-1:\vEc w] = [a,g,\vEc w]=[a-1:\vEc w^1]\in R_G,\\
1+[0:g-1:\vEc w] = [0,g:\vEc w] \in R_G .
\end{gathered}\label{eq:case-1} 
\end{equation}
The only decomposition structure  that is not considered in (\ref{eq:case-1}) is 
 $$\delta:=(0,g:\vEc w),$$ which we analyze this case below.
If $q\le g+1$, then 
$$1+[0,g:\vEc w ]=[1:g+	1-q:\vEc w]\in R_G$$
where $g+1-q\le g-1$.
Thus, we have 
\begin{theorem}
If $(G_1,G_2)=(q,1)$ and $2\le q\le g+1$, then 
$\#R_G(X)=X$.
\end{theorem}
The remaining case is proved by Proposition \ref{prop:[0,g]} below.
\begin{prop} \label{prop:[0,g]}
Let $g\ge 3$, and let  $g+2\le q\le g^2+1$.
If $[0,0:\vEc w ]\ge G_4$, then
 $1+[0,g:\vEc w]\in R_G$.
\end{prop}
 Although Proposition \ref{prop:[0,g]} can be proved by  \lq\lq borrowing\rq\rq\ from the higher entries, 
the diagrammatic representation below provides a more intuitive understanding of the underlying mechanics.
 It is a refined version of the diagram in Proposition \ref{prop:cFo-diagram}:
\begin{equation}
 \begin{tikzcd}
    \ar[swap]{d}{+1}   (0,0)   & (0,g) \ar[swap]{l}{+1}& \fbox{$(g,g-1)$}\ar[swap]{l}{+q}
	 &  (g,g-2) \ar[dotted,swap]{l}{\HS{1em}g+q}\\
     (1,0) \ar[swap,dotted]{d}{+(g-2)}&   &  &  (g,2)  \ar[swap,dotted]{u}{(g-4)(g+q)}\\
     (g-1,0) \ar[swap]{r}{+1} & \fbox{$(g,0)$}\ar[swap]{r}{+q} \ar[swap]{uul}{+q} 
     	& (0,1) \ar[swap,dotted]{r}{+g}
     	 & \fbox{$(g,1)$} \ar[dotted,swap]{u}{g+q}\\
\end{tikzcd} 
\label{diag:ext-chains}
\end{equation} 
  where the solid arrows indicate  $\lub_{\cEo}(\ep)[2,3]$, and 
  the dotted arrows indicate $\lub_{\cEo}^m(\ep)[2,3]$ with $m>1$. 
  
In fact, we seek $\ep$ that is \lex ly smaller than $\delta$ satisfying \eqref{eq:1+delta}. Because $\delta_1=0$ must hold, we show that there exists $\ep<\delta$ in $\cE$ such that
\begin{gather} 
  \ep_1 q - 1 = \delta^\circ \dG - \ep^\circ \dG. \label{eq:delta+1} 
\end{gather} 
To demonstrate this approach, consider the example where $q=g+2$.
If $\delta[2,3]=(g,1)$, which corresponds to the vertex at the lower right corner in 
(\ref{diag:ext-chains}), then we backtrack along the diagram to $\ep^\circ[2,3]:=(g-1,0)$, i.e.,
$$ \ep^\circ:=\lub_{\cEo}^{-g-2}(\delta^\circ).$$
This yields $ \delta^\circ \dG - \ep^\circ \dG= 1 + q + g = 2g+3$.
Setting $\ep_1=2$, 
we obtain
\GGG{
\ep_1q - 1  =2g+3 = \delta^\circ \dG - \ep^\circ \dG.}
Since $(\ep_1:\ep^\circ)=(2:g-1:\star)\in\cE$, we proved (\ref{eq:1+delta}).

 Let $\delta:=(0,g:\vEc w)$ be as in Proposition \ref{prop:[0,g]}.
 Then, $\delta[2,3]=(g,c)$, and 
 it
 is helpful to analyze the diagram in terms of the following sequences
 determined by the vertices $(g,x)$:
\begin{gather} 
 \begin{gathered}
 (g,a)\to (0,a+1) \to \cdots \to (g,a+1),\quad 0\le a <g-1,\\
 (g,0)\to (0,0) \to \cdots \to (g,0);
\end{gathered}  \label{eq:ext-short-chain}\\
  (g,g-1) \to (0,g)\to (0,0) \to \cdots \to (g,0).\label{eq:ext-long-chain}
\end{gather}
We refer to the two path types in \eqref{eq:ext-short-chain} as {\it extended short chains},
and the path type in \eqref{eq:ext-long-chain} as {\it the extended long chain}.
Recall the short and long chains from Definition \ref{def:chains}, which are obtained from the extended chains by removing their initial edges; consequently, the edge counts of these extended chains are $g+1$ and $g+2$,  respectively.
The   extended short chains have weight sum $g+q$, and the  extended long chain has weight sum $g+1+q$.
The initial and terminal vertices of the  extended long chain are boxed in the diagram, as is the terminal vertex of the extended short chain succeeding the long chain.
Along the long circuit, which is defined in Section \ref{sec:classification-sub}, there are $g-1$ extended short chains and one extended long chain, and 
the short circuit itself forms an extended short chain.

\begin{proof}[Proof of Proposition \ref{prop:[0,g]}]
Let   $\delta=(0,g:\vEc w)\in\cE$  sufficiently large.
 Suppose that there exists an $\ep$ satisfying \eqref{eq:1+delta}.
   Let $\gamma$ be the path in $\PathG$ determined by $\ep^\circ$ and $\delta^\circ$, 
   so $\wgt(\gamma)= \delta^\circ \dG - \ep^\circ \dG$.
Notice that $\ep < \delta $ implies
$\ep^\circ \le \delta^\circ$, and 
by Corollary \ref{cor:eval}, we have 
$\ep^\circ \dG \le \delta^\circ \dG$.
If equality holds, then \eqref{eq:delta+1} implies $\ep_1q=1$, which is impossible since $q\ge g+2$.
Thus, $ \delta^\circ \dG - \ep^\circ \dG> 0$, and hence, 
$\ep_1\ge 1$.

The diagram (\ref{diag:ext-chains}) is formed by 
joining extended chains.
Thus, $\ep^\circ[2,3]$ can be found in 
the short chain 
\GGG{
 (0,a) \to \cdots\to \ep^\circ[2,3] \to \cdots \to (g,a)\\
 \intertext{ where $0\le a<g$
 or in the long chain} 
  (0,g) \to (0,0)\to  \cdots\to \ep^\circ[2,3] \to \cdots \to (g,0).}
 Here, the vertex $\ep^\circ[2,3] $  may not coincide with the terminal vertex of the chain.
 If it  coincides, then 
   $ \ep^\circ[2,3]=(g,b)$ implies that $\ep_1=0$,
   which contradicts that $\ep_1\ge 1$.
 
 Let $(g,b)$ for $0\le b<g$ be the last vertex of the chain for both cases.
Let $y+t$ be  the weight sum of $\ep^\circ[2,3]\to\cdots\to (g,b)$  where 
$0<y\le g$, $t\in\set{0,1}$, 
and $t=1$ if and only if  $ \ep^\circ[2,3]=(0,g)$ (and hence, $y=g$).

Notice that $(g,b)$ is the initial vertex of some extended chain, and 
$\delta^\circ[2,3]$ is the terminal vertex of some extended chains.
Let $\ell$ be the number of copies of the extended chains in 
the sub-path  $\gamma_0 : (g,b) \to \cdots \to \delta^\circ[2,3]$, and 
let $s $ be the number of extended long chains contained in $\gamma_0$.
If $\gamma_0$ has two extended long chains, then the 
diagram (\ref{diag:ext-chains}) implies that
 $\wgt(\gamma)$ is at least the weight of  a long circuit, which 
contradicts Lemma \ref{lem:gamma-length}.
Thus, $s\in\set{0,1}$.
Then, $\delta^\circ\dG- \ep^\circ \dG =y + t+  \ell(g+q) + s $.

We claim that $s=t=1$ is impossible.
Notice that $s=t=1$ implies that 
the path $\gamma$ begins at the long chain and contains
an extended long chain. 
From the diagram (\ref{diag:ext-chains}), it follows that 
$\wgt(\gamma)$ is at least the weight of a long circuit, which 
contradicts Lemma \ref{lem:gamma-length}.

Notice that each extended chain contains an edge with weight $hG_1=q$.
If $\ell\ge g$,
then we have
$
gq<\ell(g+q)< \wgt(\gamma)= \ep_1q-1\le gq-1$, and hence, $\ell\le g-1$.
To construct such $\ep$, it suffices to solve the equation
\begin{gather} 
\ep_1 q - 1 = (g+q)\ell + y + u,\notag\\
\text{ i.e.,  }
(\ep_1-\ell) q = g\ell + y + u+1 \label{eq:ep-delta-equation}
\end{gather}
where $1\le \ep_1\le g$, $0\le \ell \le g-1$, $0< y\le g$,  $s+t=u\in\set{0,1}$.

\begin{prty}\label{prty:ell}  \rm The following are also satisfied:
\begin{enumerate} 
\item 
 $\lub_{\cEo}^{-\ell(g+1)-s}(\delta^\circ)[2,3]=(g,b)$; \label{cond:s}
 \item $\ep^\circ=\lub_{\cEo}^{-\ell(g+1)-s-y-t}(\delta^\circ) $;
\item
if $t=1$, then  $y=g$ and 
\begin{gather} 
\lub_{\cEo}^{-\ell(g+1) }(\delta^\circ)[2,3]=(g,0). \label{eq:s=0}
\end{gather}
\end{enumerate}
\end{prty} 

These relations provide necessary and sufficient conditions for (\ref{eq:delta+1}).
Suppose that we have values of the following parameters satisfying 
(\ref{eq:ep-delta-equation}) and Property \ref{prty:ell}:
$$ \ell,\ s,\ t,\ y,\ \ep_1.$$
 As we backtrack  along extended chains from $\delta^\circ[2,3]$ by $\ell$ as indicated in 
Property \ref{prty:ell} (1), there is at most one 
extended long chain involved. Thus, according to the diagram (\ref{diag:ext-chains}),
 the value of $s$ must be equal to 
the number of extended long chains traversed during backtracking to   $(g,b)$.
Thus, 
$s=s_\ell$ is determined by the value of $\ell$ and $\delta$. 
To be able to backtrack  
$(\ell(g+1)+s)$ times, $\delta^\circ$ must be sufficiently large, and 
$[0,0:\vEc w ]\ge G_4$ is sufficient for this purpose.
Since $\ep^\circ[2,3]$ lands in a long or short chain with $y>0$, the resulting tuple satisfies
$(\ep_1,\ep^\circ_2,\ep^\circ_3,\dots)\in\cE$ for any 
$1\le\ep_1\le g$.

Consider the case $q=g^2$ or $q=g^2+1$. 
If $q=g^2 $, then $\ell=g-1$, $\ep_1 = g$, and
$s\in \set{0,1}$ is determined by Condition (\ref{cond:s}) above.
Let $t=0$, and 
$ y:= g-u-1=g-s-1\le g-1$. Since 
 $g\ge 3$, we have $y>0$. 
The equation  (\ref{eq:ep-delta-equation}) and Property \ref{prty:ell} are satisfied.
  If $q=g^2+1 $, then $\ell=g-1$, $\ep_1 = g$,  and
$s\in \set{0,1}$ is determined by Condition (\ref{cond:s}) above.
Setting $t=0$, and $0<y:= g-u\le g$ yields a valid solution.

 Assume that $g+2 \le q \le g^2-1$, and $q=g\ell' + r'$ where $1\le \ell'\le g-1$ and $0\le r' \le g-1$.
 Consider the case $0\le r'\le 1$.  Let $0\le \ell:=\ell'-1\le g-2$, $0< \ep_1:=\ell+1\le g-1$, and $s $ is determined by 
 Condition
 (\ref{cond:s}).  
Setting $t=0$, and let   $0<y:=g+r'-u-1\le g$ yields a valid solution.
 
 Let  $r'= 2$, i.e., 
     $q=g\ell' + 2$.
Let $s\in\set{0,1}$ be such that 
\GGG{
 \lub_{\cEo}^{-(\ell'-1)(g+1)-s}(\delta^\circ)[2,3]=(g,b),\\
 \ep^*:=\lub_{\cEo}^{-(\ell'-1)(g+1)-s-y}(\delta^\circ),\quad
 0<y\le g.
 }
 Observe that $s$ is determined by $\ell'$.
We choose $\ell$ and $\ep^\circ$ depending on $s$.
Suppose $s=1$, and let $\ell:=\ell'-1$,  $0\le \ep_1:=\ell+1\le g-1$.
 Then,  setting    $0<y:=g+1-u=g$ and  
  $\ep^\circ=\ep^*$ yields a valid solution.
  
Suppose $s=0$, and  $\ep^*[2,3]$ is in a  long chain for $y>0$.
Then, let
 $\ell:=\ell'-1$ and $\ep_1:=\ell+1\le g-1$. 
Setting $t:=1$,   $y:= g$, and $\ep^\circ=\ep^*$ yields a valid solution.
Suppose $s=0$  and  $\ep^*[2,3]$ is  in a short chain for $y>0$.
Then, let $\ell:=\ell'$ and $\ep_1:=\ell+1\le g$.
Then, the diagram (\ref{diag:ext-chains}) implies that 
$$ \lub_{\cEo}^{-\ell(g+1)-s}(\delta^\circ)[2,3]=(g,b) $$
where the choice $s=0$ remains valid because backtracking one step further than $\ell'-1$ does not encounter 
a  long chain. 
Setting $t:=0$, $y:= 1 - u = 1>0$, and 
$$\ep^\circ:=\lub_{\cEo}^{-\ell(g+1) -y}(\delta^\circ)[2,3] $$
yields a valid solution.
Let $r'\ge 3$.  Let $\ell:=\ell'$ and $\ep_1:=\ell+1\le g$.  Let $s$ be determined by Condition
(\ref{cond:s}).  Then, setting $t:=0$ and $0<y:= r'-1-u \le g-2$ yields a valid solution. 
 \end{proof}

\begin{theorem}
Suppose that $q\le g^2+1$, and $g\ge 3$.
If $(G_1,G_2)=(q,1)$, then $R(X)=X + O(1)$.
If $q>g^2+1$, then $G$ has \uniexp.
\end{theorem}

\begin{proof}
Suppose that $q\le g^2+1$.
Notice that the subset $$A=\set{\ep\dG : \ep\in\cE,\ \ord(\ep)\le 3}$$ is finite, and let $G_N$ be the smallest term
that is greater than the values of $A$.
If $m\ge G_N$ and $m=\ep\dG$ is an $\cE$-expansion by $G$, then $\ord(\ep)\ge 4$.
By \eqref{eq:case-1} and Proposition \ref{prop:[0,g]}, induction on $m$ can be established with the initial value $G_N$, and 
  hence, all integers $m\ge G_N$ are $\cE$-expansions by $G$.
  This proves that $\#R(X) = X + O(1)$.
  
  Suppose that  $G$ does not satisfy \uniexp.
  We shall prove that $q\le g^2+1$.
  Let $\ep<\delta$ in $\cE$, and $\ep\dG = \delta\dG$.
  Let $\gamma$ be the path considered in Theorem \ref{thm:initial-values}.
  If $\gamma$ contains two edges $(0,g)\to (0,0)$, then it must have contained a long circuit with $g $ edges of weight $G_1$, 
  but according to Lemma \ref{lem:gamma-length}, this is impossible.
 Thus, $\gamma$ must contain at most one edge $(0,g)\to (0,0)$, at most $g-1$ edges with weight $G_1$ (Lemma \ref{lem:ell-bounds} Part (1)), i.e., $\ell\le g-1$ 
 where $\ell$ is the number described in Theorem \ref{thm:initial-values}, and hence contain at most $g-2$ chains.
 
  Recall that the short and long chains are defined in \eqref{eq:chain-1} and \eqref{eq:chain-2}, and have weight sums $g$ and $g+1$, respectively.
 We represent the chains, interspersed with edges of weight $G_1$ (corresponding to transitions of the form $(g,x) \to (0,y)$), as follows: 
 \GGG{
 w_1\to G_1 \to g \to G_1\to   \cdots \to g \to G_1 \to w_2
 }
 where at most one chain can have weight sum $g+1$ (corresponding to a long chain).
 There are $\ell-1$ chains surrounded by $G_1$, and let $t_k$ for $1\le k\le \ell-1$ be $1$ if the $k$th chain is a long one, 
 and $t_k=0$, otherwise.
 Notice that if $w_1$ or $w_2$ is coming from the long chain, and $w_k=g+1$, then 
 $t_k=0$ for all $k$; otherwise, $\gamma$ will contain a long circuit.
 Then
 \GGG{
 \wgt(\gamma)  = w_1 + w_2 + \sum_{k=1}^{\ell-1} (g +t_k)
  \le 2g+ (\ell-1) g  +1\\
  \le 2g+1+ (g-2) g  = g^2+1.
 }
By Theorem \ref{thm:initial-values}, this implies 
  $q=G_1\le m \le \wgt(\gamma)=m G_2 + \ell G_1 \le  g^2+1$.
 Therefore, we conclude that if $q > g^2+1$, then $G$ has \uniexp. 
\end{proof}

\subsection{Case II}
In this section, we  consider the case $(G_1,G_2)=(1,q)$ where $2\le q\le g$.
Since $ G_2=qG_1$, we have $$[q,0]=[0,1].$$
By Lemma \ref{lem:add-Gn}, it follows that 
\begin{equation} \label{eq:case-2}
\begin{aligned}
\vEc w\in\cEp \implies
1+[ \vEc w] &= [\lub_{\cE}(\vEc w)]\in R_G,\\
b<g-1 \implies
1+[g:b:\vEc w] &=  [g+1-q:b+1:\vEc w]\in R_G,\\
1+[g:g-1:\vEc w] &=  [g -q : \vEc w^1]\in R_G,\\ 
1+[g:0,g:\vEc w] &= [g+1-q:0:\vEc w^1]\in R_G . 
\end{aligned} 
\end{equation}

\begin{theorem} 
If $(G_1,G_2)=(1,q)$ where $2\le q\le g$,  then $R(X)=X $. 
If $q>g$, then $G$ has \uniexp.

\end{theorem}

 \begin{proof}
By (\ref{eq:case-2}), induction   with the initial value $0=[0]$ proves the statement that
all positive integers are $\cE$-expansions by $G$.
Suppose that $q>g$.
Then, $G_2 = q >g=\max\set{g,0}$.
By Corollary \ref{cor:G2-unique}, $G$ has \uniexp. 
\end{proof}

\section{The ratio functions} \label{sec:ratio}
By Theorem \ref{thm:main-1}, the following limit function is well-defined:
\begin{deF}
\rm
Let $\Nprm$ be the set of $(x,y)\in \nat^2$ with $\gcd(x,y)=1$.
Given $(x,y)\in \Nprm$, let
 $\pp(x,y):=\lim_{X\to\infty} \frac{\#R_G(X)}X$ 
where $G$ is a sequence in $\cH$ satisfying the initial values $(G_1,G_2)=(x,y)$.
We refer to $\pp$ as {\it the ratio function for $\cH$}.
\end{deF}

Let $y_0:=\max\set{g,h-1}$, and $x_0:=g+h$ if $g\le h-1$ and $x_0:=(g^2+gh+h^2)/h$ if $g\ge h$.
By Theorem \ref{thm:main-1}, if $y>y_0 $ or $x> x_0$, 
then $G$ has \uniexp, which yields $\pp(x,y) = ((g+1)+h\ome)/( y + hx\ome)$.
It is of interest to investigate
the behavior of the values of $\pp(x,y)$ on 
$$\Nprm(x_0,y_0):=[1,x_0]\times [1,y_0] \cap \Nprm.$$

As an illustration, consider  the $g$-golden ratio recurrence, where $x_0 = g^2+g+1$ and $y_0= g$.
In this case,  
  we may investigate the asymptotic formulas  of 
  \GGG{
  \frac1{\#A(g)}\sum_{(x,y)\in A(g)} \pp(x,y)\text{\quad and \quad}
  \#\set{(x,y)\in A(g) : \pp(x,y)=1}
  }
  where $A(g):=
\Nprm(g^2+g+1, g) $. 
Specifically, one might investigate whether there exist positive integers  $x_1$ and $y_1$ such that 
  $\pp(x,y)=1$ for all $(x,y) \in \Nprm(x_1,y_1)$.
  Furthermore, it is natural to ask whether these parameters $x_1$ and $y_1$ are $\Omega(g)$; that is, whether the ratios $x_1/g$ and $y_1/g$ have positive limits inferior as $g\to \infty$ (i.e., $\liminf_{g\to\infty} x_1/g > 0$ and $\liminf_{g\to\infty} y_1/g > 0$).

\newpage
\section*{Appendix}\label{sec:appendix}

\subsection*{Table 1}\label{sec:table-1}
Let $(g,h)=(2,1)$.
The list of all pairs $\ep'<\delta'$ with $\ep'\dG = \delta'\dG$ is found in the table.
\scriptsize
\begin{equation*}
\begin{array}{| c  | c | c | l | l | l | l |}
\hline
\text{Equality} & w & d & \hfill\ep'\hfill & \hfill\delta' \hfill & \hfill\text{Ratio}& \hfill\text{Total} \hfill \\
 \hline
\VS{1  em}
 \mrow{1}{2G_1= G_2}
  	& \mrow{1}{G_2} & \mrow{1}{2G_1} &  (2:  \star) & (0:  \star) & vr \ome 
  		& \mrow{1}{vr\ome} \\    
\hline
\VS{1.3em}
\mrow4{ G_1=2G_2} & \mrow1{2G_2} & \mrow1{G_1} &   (1:0: \star) & (0,2: \star) & vr \ome^2
  & \mrow5{vr(\ome+\ome^2+\ome^3)} \\
   & 2G_2+G_1  & 2G_1  &   (2:0: \star) & (0: \star) &  vr \ome^2& \\  
       &    &    &  (2:1: \star) & (0: \star) &  vr \ome^2 &  \\
                 & 2G_2   &   G_1 &  (1:0,2:\star) & (0:1:\star) & vr \ome^3& \\
                                & 2G_2   &   G_1 &  (2:0,2:\star) & (1:1:\star) & vr \ome^3& \\
 \hline
\VS{1 em}
  \mrow2{G_1=3G_2} & \mrow1{3G_2+G_1} & \mrow1{2G_1} &   (2: \star) & (0: \star) & vr \ome 
  & \mrow2{vr(\ome+ \ome^3)} \\  
      &  3G_2  &  G_1  &   (1:0,2:\star) & (0,2:\star) &  vr \ome^3& \\ 
      \hline
\VS{1.3em}
  \mrow1{2G_1=3G_2} & 3G_2 & 2 G_1 &  (2:0,2:\star) & (0,2:\star) & vr \ome^3  & vr \ome^3 \\ 
 \hline
\VS{1.3em}
 \mrow{3}{G_1=4G_2}
  	& \mrow{2}{4G_2+G_1} & \mrow{2}{2G_1} &  (2:0: \star) & ( \star) & vr \ome^2
  	 &  \mrow3{vr( \ome^2+2\ome^3)}  \\   
          &   &  &  (2:1:1:\star)  & (0,2:\star) & vr \ome^3 &   \\
	&   &  &  (2:0,2: \star)  & (0:1: \star) &vr \ome^3 &    \\ 
  \hline
\VS{1.3em}
 G_1=5G_2 
  	& 5G_2+G_1  & 2G_1  &  (2:0:1:\star) & (0,2:\star) & vr \ome^3  &
  \mrow2{vr( 2\ome^3)} 	\\   
  	&   &   &  (2:0,2:\star) & (0,2:\star) & vr \ome^3 &  \\   
 \hline 
\end{array} 
\end{equation*}

\end{document}

%% file: basic_layout_local.tex
\newcommand\HS[1]{\mbox{\rule{#1}{0pt}}}
\newcommand\VS[1]{\mbox{\rule{0pt}{#1}}}

%% file: MathMacro.tex
\renewcommand{\implies}{\Rightarrow}

\newcommand{\zz}{\mathbb{Z}}
\newcommand{\nat}{\mathbb{N}}

\newcommand{\dis}{\displaystyle}
\newcommand{\ncom}{\newcommand}

\newcommand{\To}{\longrightarrow}

\def\abs#1{\left\vert #1 \right\vert}

\def\set#1{\lbrace #1 \rbrace}

\newcommand{\Inv}{{}^{-1}}

\newcommand{\al}{\alpha}
\newcommand{\ome}{\omega}
\newcommand{\ep}{\epsilon}
 \ncom{\Del}{\Delta}

\newcommand{\GGG}[1]{
\begin{gather*}
#1
\end{gather*}
}
\newcommand{\AAA}[1]{
\begin{align*}
#1
\end{align*}
}

%% file: algebra.tex
\newcommand{\ord}{\operatorname{ord}} 

%% file: Zec_local-2nd.tex
\newcommand{\zec}{Zeckendorf}

\newcommand{\fib}{Fibonacci}
\newcommand{\funds}{fundamental sequence}

\newcommand{\HSW}[1]{\rule{#1\linewidth}{0pt}}

 \newcommand{\seq}[1]{\set{#1_k}_{k=1}^\infty}

\newcommand{\cE}{\mathcal{E}}
\newcommand{\cEo}{\cE^\circ}
\newcommand{\cEbar}{\overline{\cE}}
\newcommand{\cEstar}{\cE^\star}
\newcommand{\cEp}{\cE_{\mathrm p}}
\newcommand{\cEm}{\cE_{\mathrm m}}

\newcommand{\cH}{\mathcal{H}}

\newcommand{\lex}{lexicographical}

\newcommand{\cf}{coefficient function}

\newcommand{\wt}{\widetilde}
\newcommand{\len}{\mathrm{len}}

\newcommand{\nN}{n \in \nat}
\newcommand{\kN}{k \in \nat}

\newcommand{\baR}{\overline}

\newcommand{\uniexp}{the unique expansion  property under $\cE$}

\newcommand{\opi}{of positive integers} 
\newcommand{\roe}{rule of expansions}

\newcommand{\eval}{\mathrm{eval}}

\newcommand{\lub}{\mathrm{lub}}
\newcommand{\cT}{\mathcal{T}}
\newcommand{\fB}{\mathfrak{B}}
\newcommand{\fT}{\mathfrak{T}}
\newcommand{\Path}{\mathrm{Path}}

\newcommand{\wgt}{\mathrm{wt}}

\newcommand{\thepaths}{ \Path(\Gamma_G) }
\newcommand{\mrow}[2]{\multirow{#1}{*}{$#2$}}

\newcommand{\appone}{Appendix: Table 1}

\newcommand{\expanD}{expanded}

\newcommand{\dG}{\cdot G}
\newcommand{\dH}{\cdot H}

\newcommand{\Tp}{T_{\mathrm p}}
\newcommand{\Tm}{T_{\mathrm m}}
\newcommand{\Sp}{S_{\mathrm p}}
\newcommand{\Sm}{S_{\mathrm m}}
\newcommand{\spb}{semi-proper block}
\newcommand{\NN}{\mathfrak{N}}

\newcommand{\NNB}{\NN^\infty(B)}
\newcommand{\vEc}{\mathbf}
\newcommand{\PathG}{\Path(\Gamma_G)}

\newcommand{\pp}{\mathfrak{p}}
\newcommand{\Nprm}{\NN_2^{\mathrm{\rm prm}}}
\newcommand{\theFigures}{Figures
 \ref{fig:g>=h},
\ref{fig:g=h-1}, and \ref{fig:g<h-1}}

\newcommand{\ecl}[1]{\llbracket #1\rrbracket}

\newcommand{\ShowProof}[1]{
  \ifthenelse{\equal{\ProofInd}{ON}}{
  #1
  }{}}